\RequirePackage{fix-cm}
\documentclass[smallextended]{svjour3}       % onecolumn (second format)
\smartqed  % flush right qed marks, e.g. at end of proof
\usepackage{amssymb}
\usepackage{latexsym}
\usepackage{url}
\usepackage{amsmath}
\usepackage{verbatim}
\usepackage{graphicx}
\usepackage{epsf}
\usepackage{enumerate}
\usepackage{geometry}
\usepackage[width=6cm]{caption}
\usepackage{hyperref}
\usepackage{amsfonts}
\usepackage{booktabs}%
\providecommand{\U}[1]{\protect\rule{.1in}{.1in}}
\providecommand{\U}[1]{\protect\rule{.1in}{.1in}}
\hypersetup{breaklinks=true,
pagecolor=white,
colorlinks=true,citecolor=black,filecolor=black,linkcolor=black,urlcolor=black
}
\newtheorem{alg}{Algorithm}
\newtheorem{condition}{Condition}
\numberwithin{equation}{section}

\newcommand{\R}{\mathbb{R}}

\newcommand{\N}{\mathbb{N}}

\begin{document}

\titlerunning{Zero-convexity, perturbation resilience, and subgradient projections 
for feasibility-seeking methods}
\title{Zero-Convex Functions, Perturbation Resilience, and Subgradient Projections for Feasibility-Seeking Methods} 
%\thanks{Grants or other notes
%about the article that should go on the front page should be
%placed here. General acknowledgments should be placed at the end of the article.}
%\titlerunning{Short form of title}        % if too long for running head

\author{Yair Censor \and Daniel Reem         
}

%\authorrunning{Short form of author list} % if too long for running head

\institute{Yair Censor \at
              Department of Mathematics, University of Haifa, Mount Carmel, Haifa 3498838, Israel. \\
              \email{yair@math.haifa.ac.il}           %  \\
%             \emph{Present address:} of F. Author  %  if needed
           \and
           Daniel Reem \at
           This work was done while the author was at 
            the Department of Mathematics, University of Haifa, Haifa, Israel (2010-2011), 
            and at IMPA - Instituto Nacional de Matem\'atica Pura e Aplicada,
   Rio de Janeiro, Brazil (2011-2013). Current Address: Instituto de Ci\^encias Matem\'aticas e de Computa\c{c}\~ao (ICMC), University of S\~ao Paulo at S\~ao Carlos, Avenida Trabalhador S\~ao-carlense, 400 - Centro, 
CEP: 13566-590, S\~ao Carlos,  SP, Brazil. \\
\email{dream@icmc.usp.br}
}

\date{Received: November 18, 2012 / Revised: January 6, 2014 / Accepted: May 2, 2014}
% The correct dates will be entered by the editor

\maketitle    

\begin{abstract}
The convex feasibility problem (CFP) is at the core of the modeling of many problems in
various areas of science. 
Subgradient projection methods are important tools for solving the CFP 
 because they enable the use of subgradient calculations 
instead of orthogonal projections onto the individual sets of the problem. 
 Working in a real Hilbert space,  we show that the sequential subgradient projection 
method is perturbation resilient. By this we mean that under appropriate conditions 
the sequence generated by the  method converges weakly, and sometimes also strongly, 
to a point in the intersection of the given subsets of the feasibility problem, despite 
certain perturbations which are allowed in each iterative step. Unlike previous 
works on solving the convex feasibility
problem, the involved functions, which induce the feasibility problem's
subsets, need not be convex. Instead, we allow them to belong to a wider and
richer class of functions satisfying a weaker condition, that we call
\textquotedblleft zero-convexity\textquotedblright. This class, which 
is introduced and discussed here, holds a promise to solve optimization 
problems in various areas, especially in non-smooth and non-convex 
optimization. The relevance of this study to approximate 
minimization and to the recent 
superiorization methodology for constrained optimization is explained.


% \PACS{PACS code1 \and PACS code2 \and more}
% \subclass{MSC code1 \and MSC code2 \and more}
\end{abstract}

\section{Introduction}
\label{sec:Intro}

\subsection{Feasibility problems}  
In this paper we investigate, among other things, perturbation resilience of the 
sequential subgradient projection (SSP) method for feasibility-seeking.  
Feasibility-seeking is concerned with solving the
\textit{convex feasibility problem} (CFP), which is, to find a point in the
intersection $C=\cap_{j}C_{j}$ of a family (usually finite) of closed convex
subsets $C_{j}\subseteq\mathbb{R}^{d}$ of the Euclidean space or of a real Hilbert
space. The CFP formalism is at the core of the modeling of many problems in
various areas of mathematics and the physical sciences, among them image
reconstruction, radiation therapy treatment planning, data compression, and
antenna design. See, e.g., \cite{bb96,cccdh11,Combettes1996} for references.
One of the reasons for this is the observation that the solution of a system
of inequalities is nothing but a point in the intersection of the level-sets
of the corresponding functions which induce these inequalities. In particular,
when convex functions are considered, the context is that of the CFP. Feasible
sets represented by a system of inequalities appear frequently in optimization
\cite{BertsekasNedicOzdaglar2003,BorweinLewis-book-2006,Dixit1976,HenrLassCSM2004,Razumichin1987,Rockafellar1970,Tuncel2010}.

\subsection{Perturbation resilience} Perturbation resilience asks how, and by how
much, can the iterates of an algorithm be perturbed at each iterative step
without losing the overall convergence to a solution of the original problem.
Stability of algorithms is a well-known topic in numerical analysis of
algorithms, see, e.g., \cite{BtEgN2009,BS00,higham96}. However, this is
commonly studied in the context of supplying a guarantee that an algorithm
that has such stability is immune to changes that occur in its progress due to
noise, errors, and other disturbances that can cause the algorithm to deviate
from its \textquotedblleft pure\textquotedblright\ mathematical formulation.

Our motivation in studying perturbation resilience comes not only from this classical
context, but also from the recent line of 
research of a new concept called \textit{superiorization.} 
The superiorization principle aims not at finding a feasible point (the
feasibility problem) and not at the quest for a constrained minimum point.
Instead, the declared aim is to seek a feasible point that is
\textquotedblleft better\textquotedblright, i.e., \textit{superior,} over
other reachable feasible points, with respect to a given objective 
function. Superiorization algorithms rely on bounded 
perturbation resilience that gives the user the certificate to perturb the
iterations of an efficient feasibility-seeking method in a way that will steer
the iterates toward a superior solution without losing the guarantee of
convergence to a feasible point. See 
\cite{ButnariuDavidiHermanKazantsev,CensorDavidiHerman,CDHST2013,dhc09,HermanDavidi,hgdc12,scottetal10}
and \cite{DavidiPhD} for more details and  for experimental work 
demonstrating that algorithms can efficiently and usefully perform 
superiorization. 

An additional aspect of perturbation resilience 
 is a greater flexibility that the users of a given algorithm may have. 
 Indeed, once it is proved that the algorithm is perturbation resilient,    
 the users have more freedom  in generating the 
iterative sequence and, in particular, may obtain faster 
convergence by selecting appropriately the perturbation terms. 

\subsection{Subgradient projection methods}\textit{ }The reason for
investigating perturbation resilience of \textit{subgradient projection 
methods}, such as the \textit{cyclic} \textit{subgradient projection} (CSP) 
method of \cite{cl82}, is their advantage in feasibility-seeking. 
Under the commonly used assumption that each of the sets $C_j$ of the CFP 
can be written as the zero-level-set of 
some convex function $g_j$, $j\in J$, namely $C_j=\{x|\,\,g_j(x)\leq 0\}$ 
(as happens in the case of convex inequalities), the advantage is that instead of orthogonal (least Euclidean
distance) projections onto the sets $C_j$, commonly employed
by many other feasibility-seeking algorithms, the subgradient projection 
methods use \textquotedblleft subgradient 
projections\textquotedblright\ . When each set $C_j$ is 
linear (i.e., hyperplanes or half-spaces) or otherwise \textquotedblleft
simple\textquotedblright\ to orthogonally project onto (like balls), then
there is no advantage in using subgradient projections. But in other cases the
subgradient projections are easier to compute than orthogonal projections
since they do not call for the, computationally demanding, inner-loop of least
Euclidean distance minimization, but rather employ the \textquotedblleft
subgradient projection\textquotedblright\ which is merely a step in the 
negative direction of a calculable subgradient of $g_j$ at the current 
iteration; see, e.g.,
\cite{ButnariuCensorGurfilHadar,CensorLent,CensorZenios,IusemMoledo}.  
For a general review on projection
algorithms for the CFP see \cite{bb96} and consult the recent work 
\cite{Cegielski2012}.

\subsection{Current literature} Perturbation resilience of algorithms in
optimization is discussed, under the title of stability, in \cite{BS00} but
many algorithms still await investigation of this feature. The relevant
discussions in \cite{ButnariuDavidiHermanKazantsev,BRZ_InexactBregman,ButnariuReichZaslavsky,Combettes2001,CorvellecFlam,IusemOtero,Kiwiel2004,NedicBertsekas2010,OstrowskiStability,PRZ2008,PRZ2009,SolodovZavriev1998} are about
feasibility-seeking projection methods or about the incremental 
method that use orthogonal projections whose nonexpansivity (or related
properties) often plays an important role in the convergence proofs. 
Since subgradient operators are usually not nonexpansive, proofs of convergence 
of the corresponding methods should use different properties.

Currently available theorems on perturbation resilience of iterative
feasibility-seeking projection methods are for methods that employ orthogonal
(least Euclidean distance) projections onto convex sets. To the best of our
knowledge, with the exception of the work of De Pierro and Iusem  \cite{DePierroIusem1988} 
and of Combettes \cite{Combettes2001}, perturbation 
resilience of the subgradient projection method for solving the feasibility
problem has not been dealt with in the literature. 

The perturbations  
considered in \cite{DePierroIusem1988} are different from those that we consider.  
The setting is a finite-dimensional space, convex functions, almost 
cyclic control, and a Slater-type condition is imposed on the functions $g_j$ which 
induce the subsets $C_j$. 

The work of Combettes describes a 
general framework for dealing with some optimization algorithms involving a
generalization of Fej\'{e}r-monotonicity in their convergence analysis, in
which perturbations of the type we consider are allowed 
\cite[Section 4]{Combettes2001}. However, 
neither our Theorem \ref{thm:resiliency} follows from \cite{Combettes2001} nor
do the results of \cite{Combettes2001} follow from ours (e.g., because, on 
the one hand Combettes considers 
only convex functions, while we allow more general functions, but on the other hand, 
he also considers operators beyond the subgradient operator for 
convex functions, such as nonexpansive operators). Nonetheless, 
Theorem \ref{thm:resiliency} below generalizes 
the related result \cite[Corollary 6.10(i)]{BauschkeCombettes2001} from 
the setting of convex functions without perturbations to zero-convex functions 
with perturbations. 

A common assumption in many works regarding the feasibility problem is the
convexity of the functions whose level-sets define the subsets $C_{j}$ (thus
the name CFP). When this assumption is removed, the corresponding convergence
results are quite weak (local convergence or convergence of subsequences) see,
e.g., \cite{CorvellecFlam}. The only strong (global, but without
perturbations) convergence result that we are aware of is
\cite{CensorSegal2006} in which the convexity is replaced by the concept of
quasiconvexity (i.e., $f(\alpha x+(1-\alpha)y)\leq\max\{f(x),f(y)\}$ for all
$x,y$ and all $\alpha\in\lbrack0,1]$) along with a strong continuity condition
(H\"{o}lder or Lipschitz) of the involved functions; the setting there is a
finite-dimensional Euclidean space and the algorithm is a kind of a 
subgradient projection method (with star-subdifferentials \cite{Penot1998}).

\subsection{The class of zero-convex functions}\label{subsec:ZeroConvex} A variant of our method, namely
the cyclic subgradient projection (CSP) method (for functions defined on the
whole space), was previously discussed in \cite{cl82}, \cite[Theorem 
5.3.1]{CensorZeniosBook} in a finite-dimensional Euclidean space, for finitely
many convex functions and without perturbations. See also \cite{bb96} for a
Hilbert space treatment. In contrast, the nonconvex functions that we consider
here are functions which satisfy a generalized version of the subgradient
inequality. We call these functions zero-convex. An equivalent 
characterization of these functions (when they are lower semicontinuous) is
that their zero-level-sets are convex: see 
Proposition \ref{prop:0ConvCharacterize}\eqref{item:LevelSet} below.

Since a well-known characterization of quasiconvex functions is the property that all their
$\beta$-level-sets $\{x\mid f(x)\leq\beta\}$, $\beta\in\mathbb{R}$, are convex 
\cite[pp. 135--136]{bss-3ed-2006}, 
it follows, in particular, that when they are lower semicontinuous, then 
they are zero-convex, and hence the class of zero-convex functions is quite wide. 
Zero-convex functions may lack properties 
that convex functions have and their standard subdifferential might be empty
at many points. In return, their corresponding $0$-subdifferential is never empty.

The class of zero-convex functions holds a promise for studying optimization
problems which involve non-convex functions and to enrich the theory 
of generalized convexity \cite{ADSZ2010,CambiniMartein2009-book,CM-LV-1998-handbook,HKS2005-handbook}. 
The subclass of nonconvex (multivariate) polynomials seems to be of special interest. 
An example are polynomials which appear in the context 
of control theory \cite{HenrLassCSM2004}. As said there (page 72): \textquotedblleft Polynomial
optimization problems arising from control problems are often highly
non-convex, with several local optima, and are difficult to
solve...\textquotedblright. Additional related discussion can be found in
\cite[Problems 1 and 2]{HenrLassTAC2012}, with 2-variable polynomials whose degree tends to
infinity, and in \cite{HenrLass-Bookchap-2005,LassSIOPT2001}.  A 
related example is Example \ref{ex:polynomial} below. Zero-convex 
functions can help to analyze systems of (multivariate polynomial) equations,
much like convex optimization helps doing so in other cases 
\cite{CGTV-IJRNC-2003}. They can help in the analysis of (quasiconvex) quadratic functions 
which appear in the context of economics \cite[Chapter 6]{ADSZ2010}, 
\cite[Chapter 6]{CambiniMartein2009-book}. 
Our method (Algorithm \ref{alg:perturbed-csp} below) can be used for 
accelerating convergence in the case of quasiconvex polynomials \cite{HildebrandKoppe2013}. 

As said above, lower semicontionuous quasiconvex functions are zero-convex. Hence this subclass of zero-convex functions is promising too, 
especially when taking into account that such functions arise in optimization \cite{CM-LV-1998-handbook,HKS2005-handbook,Martinez-Legaz1988} or related areas such as economics and operations research \cite{ADSZ2010,CambiniMartein2009-book}, 
location theory \cite{Gromicho1998-book}, control  \cite{BGJ2013,BarronLiu1997}, and 
geometric problems \cite{AmeBerEpp-Algs-1999,Epp-MSRI-2005,Epp-TALG-2004-qaba}. 
In this context see Example \ref{ex:Voronoi} and Example \ref{ex:VoronoiWeighted} 
below where the involved 
(geometric) function is not necessarily quasiconvex. See also Section \ref{sec:ComputationalResults} below. Functions which appear 
in global optimization  \cite{HP1995-handbook,HorstTuy-book-1990,Pinter-book-1996} 
seem to be of interest too since they are usually nonconvex, e.g., d.c. functions (namely 
functions which can be represented as a difference of two convex functions).

\subsection{The number of involved sets} In most works  
dealing with subgradient  
projection methods for solving the CFP, a common assumption 
is that the feasible set $C$ 
is obtained from the intersection of finitely 
many sets.  
However, because infinitely many sets 
do appear in theory and practice, e.g., when dealing with infinite 
systems of linear equalities \cite{GG-book-1981} 
or with infinitely many nonlinear (convex) constraints arising in certain 
problems in economics and other areas (see \cite[pp. xiii-xiv]{ButnariuIusemBook} 
and the references therein), it is natural to consider also the CFP with 
infinitely many sets appearing in the formulation of the problem, and this is done in the present paper. A few other works considering  
 the CFP with infinitely many sets exist, for instance,  
 \cite{BauschkeCombettes2001,BCK2006,Combettes2001}  and  \cite{ButnariuCensorReich1997,ButnariuIusemBook}, but some 
do not consider the SSP. 

\subsection{The contributions of the present paper} 
The contributions of the present paper are listed as follows: 
(1) Introducing and discussing in a quite detailed way the class of 
zero-convex functions, a rich class of convex and nonconvex functions 
which holds a promise to solve optimization problems in various areas, 
especially in non-smooth and non-convex optimization;
(2) Discussing  the  sequential subgradient projection method 
for solving the feasibility problem, where the involved functions are zero-convex 
functions defined on a closed and convex subset of a real Hilbert space;
(3) Showing that certain perturbations are allowed without losing the 
weak and global convergence of sequences, generated with 
such perturbations, to a solution of the feasibility problem; 
(4) Sometimes the convergence is in norm; 
(5) The control sequence, according to which the subsets are employed during 
the sequential iterative process, can be more general than the cyclic or 
almost cyclic (quasi-periodic) controls;  
(6) Our results apply to feasibility problems with finitely- or infinitely-many sets; 
(7) Our results can be applied to additional optimization schemes (approximate minimization, 
superioization).

\subsection{Paper layout} The paper is laid out as follows. In Section
\ref{sec:zero-convex} the zero-convex functions and $0$-subdifferentiabilty are
defined and a few examples are given. In Section \ref{sec:properties} some of
their properties are discussed. The algorithm is formulated in Section
\ref{sec:alg}. Additional conditions for its convergence are listed in
Section \ref{sec:conditions} and its convergence is analyzed in Section
\ref{sec:convergence}. In Section \ref{sec:ComputationalResults} we 
present some computational results. We end the paper in Section \ref{sec:FurtherDiscussion} 
with a discussion of a number of issues related to the main themes of this paper, 
as well as several lines for further investigation.

\section{Zero-convex functions: Definition and examples\label{sec:zero-convex}}

In this section we introduce the class of functions that we deal with in this
paper and illustrate it with examples. These functions satisfy a generalized
version of the subgradient inequality described in Definition  
\ref{def:0subdifferential} below. 

From now on, unless otherwise stated, $H$ is a real Hilbert space 
with an inner product $\langle\cdot,\cdot\rangle$ and a norm $\|\cdot\|$,  
and $\Omega$ is a nonempty and convex subset of $H$ (closed in many cases). The 
$\beta$-level set of a function $g:\Omega\rightarrow\mathbb{R}$ is the set
$g^{\leq\beta}:=\{x\in\Omega\mid\,g(x)\leq\beta\}$ and, in particular, the
zero-level-set is $g^{\leq0}=\{x\in\Omega\mid g(x)\leq0\}$. The distance (or the gap) 
between a point $x\in H$ and a set $A\subseteq H$ is $d(x,A)=\inf\{d(x,a)\mid
a\in A\}$. The line segment connecting two points $x_1,x_2\in H$ is the set  $[x_{1},x_{2}]:=\{x_1+t(x_2-x_1)\,|\, t\in [0,1]\}$. 

\begin{definition}
\label{def:0subdifferential} Let $H$ be a real Hilbert space . Let $\Omega$ be a nonempty  convex subset of
$H$. A function $g:\Omega\rightarrow\mathbb{R}$ is said to be
\texttt{zero-convex at the point }$y\in\Omega$ if there exists a vector $t\in
H$ (called a $0$\texttt{-subgradient of }$g$\texttt{ at }$y$) satisfying%
\begin{equation}
g(y)+\langle t,x-y\rangle\leq0, \quad\forall x\in g^{\leq0}%
.\label{eq:0subgradient}%
\end{equation}
When the corresponding vector $t$ is given, then $g$ is said to be
zero-convex\texttt{ at }$y$\texttt{ with respect to }$t$. The set of all
$0$-subgradients of $g$ at $y$ is denoted by $\partial^{0}g(y)$ and called 
the 0-subdifferential of $g$ at $y$. A function
$g$ satisfying \eqref{eq:0subgradient} for all $y\in\Omega$ will be called
zero-convex\texttt{ on }$\Omega$ or just zero-convex (or 0-convex).
\end{definition}

As the examples below show (see Section \ref{sec:ComputationalResults} for 
additional examples), zero-convex functions are not necessarily convex. Also, by taking in
\eqref{eq:0subgradient} the vector $t$ to be in the dual space, the definition
can be extended to any real normed space and even beyond (e.g., locally convex
topological vector spaces and even to linear spaces if $t$ is merely a 
possibly discontinuous linear functional). However, we confine ourselves to
real Hilbert spaces.

\begin{remark}
\label{rem:Geometric0-Convex} {\bf Geometric interpretations: } 
The zero-convexity of a 
function $g$ can be illustrated geometrically. Two such interpretations are given below. \\

{\noindent \bf First interpretation: using the graph: } 
See Figures \ref{fig:ZeroConvexityIllustration1} and \ref{fig:ZeroConvexityIllustration2}. 
In what follows, it is useful to adopt the 
following terminology: the $g$\texttt{-nonpositive part of the graph of a
function} $f:\Omega\rightarrow\mathbb{R}$ is the set $\{(x,f(x))\mid x\in
g^{\leq0}\}$. Using this notion, one can see that the function $g$ is
zero-convex at $y$ with respect to $t$ if the $g$-nonpositive part of the
graph of the affine function $f(x)=g(y)+\langle t,x-y\rangle$ is below $0$.
Therefore, in order to check whether $g$ is zero-convex at $y$ with respect to
the vector $t$, we draw the graphs of this $f$ and of $g$, then we remove from
the domains of definition of these graphs all the points $x$ for which $g$ is
positive, and then we check whether the remaining part of the graph of $f$ is
below $0$. \\

%%%%%%%%%%%%%%%%%%%%%%%%%%%%%%%%%%%%%%%%%%%%%%%%%%%%%%%%%%%%%%%%%%%%
\begin{figure}
\begin{minipage}[t]{0.48\textwidth}
\begin{center}
{\includegraphics[clip,scale=0.63]{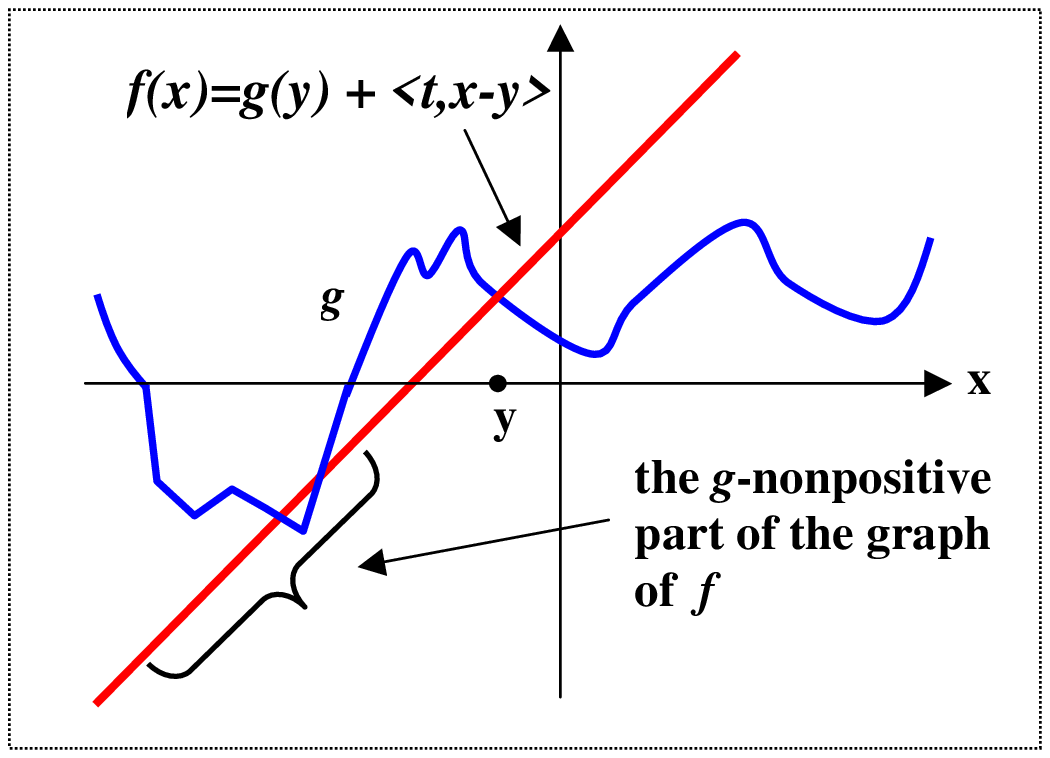}}
\end{center}
 \caption{The first geometric interpretation of zero-convexity: using 
 the graph (Remark \ref{rem:Geometric0-Convex}).} 
\label{fig:ZeroConvexityIllustration1}
\end{minipage}
%\hfill
\begin{minipage}[t]{0.48\textwidth}
\begin{center}
{\includegraphics[clip,scale=0.63]{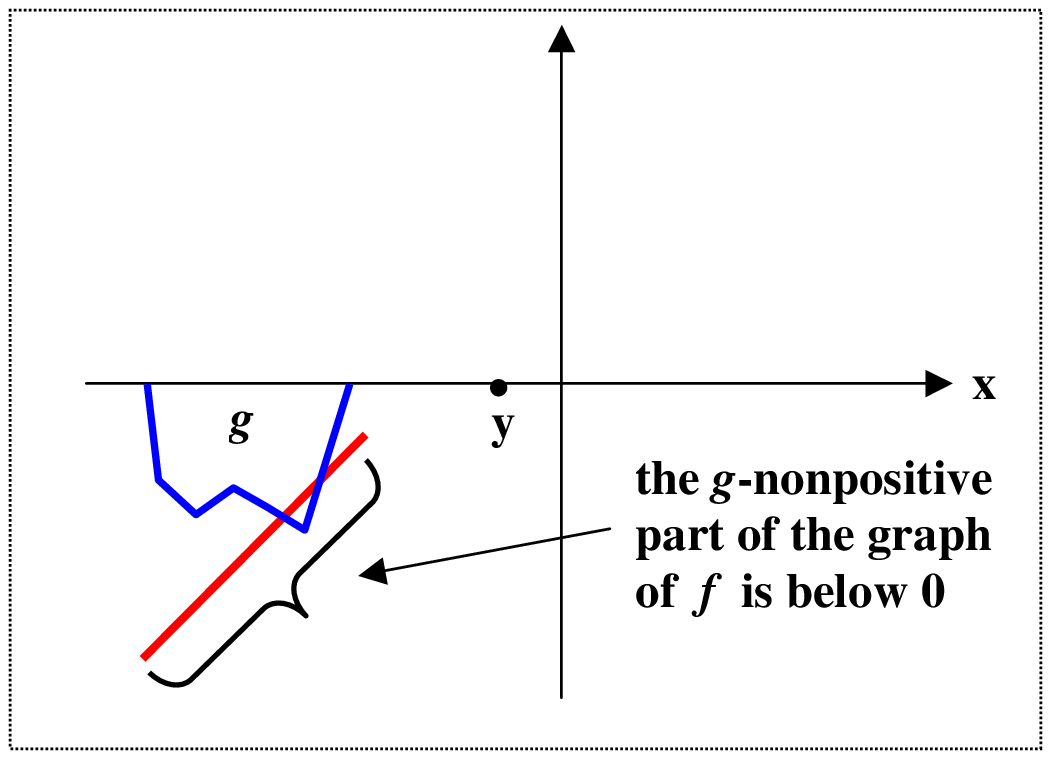}}
\end{center}
 \caption{The setting of Figure \ref{fig:ZeroConvexityIllustration1} after the $g$-positive parts of $f$ and $g$ were removed.} 
\label{fig:ZeroConvexityIllustration2}
\end{minipage}
\end{figure}
%%%%%%%%%%%%%%%%%%%%%%%%%%%%%%%%%%%%%%%%%%%%%%%%%%%%%%%%%%%%%%%%%%%

{\noindent \bf Second interpretation: using separating hyperplanes: } 
This interpretation holds only when $y\notin g^{\leq 0}$. 
We assume also that $g^{\leq 0}\neq \emptyset$. 
See Figure \ref{fig:0ConvHyperplane}. In this case 
\eqref{eq:0subgradient} implies that if $g$ is zero-convex 
at $y$ with a 0-subgradient $t$, then $t\neq 0$ (otherwise $g(y)\leq 0$ because 
of \eqref{eq:0subgradient}, a contradiction) and  for each $\omega \in (0,1]$ 
the hyperplane 
\begin{equation}
M(t,\omega):=\{x\in H|\,\, \langle t,x-y\rangle=-\omega g(y)\}
\end{equation}
strictly separates $y$ from $g^{\leq 0}$. On the other hand, as proved  
Proposition \ref{prop:0ConvCharacterize}\eqref{item:LevelSet} below, 
if $g^{\leq 0}$ is closed and convex, then $g$ is zero-convex at each point 
and any (closed) hyperplane separating $y\notin g^{\leq 0}$ from  $g^{\leq 0}$ (including  $M(t,\omega)$) allows us to find a 0-subgradient $t\in \partial^0 g(y)$ 
and to express it explicitly. In fact, any multiplication of this $t$ by a scalar 
greater than 1 remains a 0-subgradient as follows from 
Proposition \ref{prop:0Properties}\eqref{item:fg} below. 
Thus, at least when $g^{\leq 0}$ is nonempty, 
closed and convex, there is a certain duality between the 0-subgradients 
of $g$ at points $y\notin g^{\leq 0}$ and (closed) separating hyperplanes 
between $g^{\leq 0}$ and these points $y$. The freedom in the choice 
of the separating hyperplane yields a freedom in the choice of $t$, and this freedom may help 
in practice.   
\end{remark}

%%%%%%%%%%%%%%%%%%%%%%%%%%%%%%%%%%%%%%%%%%%%%%%%%%%%%%%%%%%%%%%%%%%
\begin{figure}
%\hfill
\begin{minipage}[t]{1\textwidth}
\begin{center}
{\includegraphics[clip,scale=0.7]{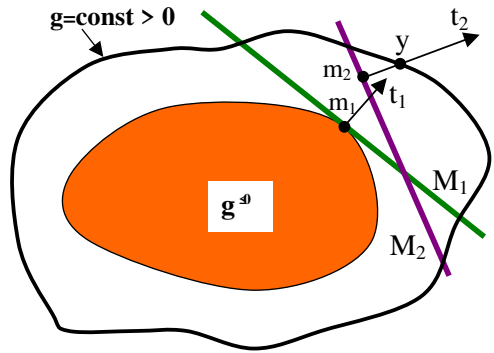}}
\end{center}
 \caption{The second geometric interpretation of zero-convexity: using 
 separating hyperplanes (Remark \ref{rem:Geometric0-Convex}). The 0-subgradients 
 can be expressed explicitly using \eqref{eq:t_remark}.} 
\label{fig:0ConvHyperplane}
\end{minipage}
\end{figure}
%%%%%%%%%%%%%%%%%%%%%%%%%%%%%%%%%%%%%%%%%%%%%%%%%%%%%%%%%%%%%%%%%%%

\begin{remark}
\label{rem:OtherSubdifferentials} To the best of our knowledge, our
generalizations of the subgradient inequality and the subdifferential in
Definition \ref{def:0subdifferential} are new. Several other generalizations or
variations of the standard notion of subdifferential have been considered in
the literature, e.g., the Clarke subdifferential \cite[pp. 25--27]%
{Clarke1983}, \cite{CSLW1998}, the Fr\'{e}chet and Hadamard subdifferentials
\cite{Rockafellar1979}, the $G$-subdifferential \cite{Ivanov2004}, 
the $H$-subdifferential \cite{Martinez-Legaz1988}, Mordukhovich's Subdifferential \cite{Mordukhovich1976,Soleimani-damaneh2010},
Plastria's lower subdifferential \cite{Plastria}, the Quasi-subdifferential
\cite{GreenbergPierskalla}, the Q-subdifferential \cite{MartinezLegazSach}, 
the $\Phi$-subdifferential \cite{PallaschkeRolewicz1997}, the star-subdifferential \cite{Penot1998},  the $\epsilon$-subdifferential \cite{MonteiroSvaiter2013}, 
generalizations of the subgradient inequality such as the notion of invexity 
\cite{Ben-IsraelMond1986,Hanson1981} or other notions related to convexity 
such as approximate convexity \cite{DJL2009,NgaiLucThera}. For a survey on
some of these concepts see \cite{BorweinZhu1999}.
\end{remark}

\begin{remark}\label{rem:0SubgradCompute} 
Computation of $\partial^{0}g$ is not always a simple task but we do have a 
theoretical method which enables the computation of an element in
$\partial^{0}g(y)$ for each $y\in\Omega$ whenever $g^{\leq 0}$ is closed and convex.  
The method is as follows. If $y\in g^{\leq 0}$, then we 
simply take $t=0$. If $y\notin g^{\leq 0}$, then we can take 
\begin{equation}\label{eq:t_remark}
t=\frac{g(y)}{\Vert y-m\Vert^{2}}(y-m), 
\end{equation}
where $M$ is any (closed) hyperplane which separates $y$ from $g^{\leq0}$ 
and $m\in M$ is the orthogonal projection of $y$ onto $M$. See Figure 
\ref{fig:0ConvHyperplane} above for an illustration and 
Proposition \ref{prop:0ConvCharacterize}\eqref{item:LevelSet} below for a proof. 

The examples given in this section, together with the propositions and their proofs given 
in Section \ref{sec:properties} and the computations given in Section \ref{sec:ComputationalResults}, illustrate further some of the techniques of computation.
In this connection we note that if one knows how to compute $G(y):=d(y,g^{\leq 0})$, then this 
yields a convex function whose 0-level-set coincides with $g^{\leq 0}$, and at least 
for the purpose of the CFP, one may want to use $G$ instead of $g$. 
However, as already said in Section \ref{sec:Intro}, usually this computation 
is not simple, and, in addition, it may result in either a complicated 
function $G$ or complicated (standard) subgradients. 
 Nevertheless, if $G$ can be computed, then one also has an additional 
way to compute 0-subgradients of $g$ (see Proposition  \ref{prop:0ConvCharacterize}\eqref{item:tm}) 
and this freedom may help in practice. 
\end{remark}

\begin{example}
\label{ex:convex} Any convex function $g:H\to\R$ having at least one point of continuity is 
zero-convex at any $y\in H$. This is so because in this case \cite[p. 76]{VanTiel1984} it has 
a standard subgradient at $y$ and the standard subgradient inequality 
\begin{equation}
g(y)+\langle t,x-y\rangle\leq g(x) \label{eq:SubgradientConvex}%
\end{equation}
implies that $g(y)+\langle t,x-y\rangle\leq0$ whenever $x\in g^{\leq 0}$, that is, 
\eqref{eq:0subgradient} holds with a standard subgradient $t\in \partial g(y)$. 
In particular $g$ is zero-convex at any $y\in H$ whenever $H=\R^n$ because 
by \cite[p. 70]{VanTiel1984} the finite dimensionality of $H$ implies that $g$ 
is continuous everywhere. 

In general, 
whenever $g:\Omega\to \R$ has a standard subgradient $t$ at some $y\in\Omega$, then 
$t$ is a 0-subgradient of $g$ no matter what subset is $\Omega$. This is true even if 
$g$ is not convex but \eqref{eq:SubgradientConvex} holds. In this connection, 
 Corollary \ref{cor:Quasiconvex} below implies that any lower semicontinuous  quasiconvex 
 function is zero-convex. 
\end{example}

\begin{example}
\label{ex:nonpositive} Any nonpositive function $g$ is zero-convex at every $y\in\Omega$ 
with $t=0$. However, this class of functions is not interesting for our SSP 
algorithm (Section \ref{sec:alg} below) since in this case any initial point $y$
will satisfy $g_{j}(y)\leq0$ for all involved functions $g_{j}$, hence the
generated sequence will be constant (equal to $y$ itself) which obviously
converges to a point in the intersection $C=\cap_{j\in J}\{x\in\Omega\mid g_{j}%
(x)\leq0\}$. Additionally, any positive function is zero-convex since
\eqref{eq:0subgradient} is void. But again, this is not interesting for our algorithm. 
However, a nonnegative function having a unique root (like many energy functions) 
is interesting for our algorithm since it is zero-convex (because its 
zero-level-set is obviously closed and convex: see Remark \ref{rem:Geometric0-Convex}, 
second interpretation) 
and hence, when we apply our algorithm to it, we can find its root, which is also 
its unique minimum. 
\end{example}

\begin{example}
\label{ex:sin} Let $g:\mathbb{R}\rightarrow\mathbb{R}$ be defined by%
\begin{equation}
g(x):=\left\{
\begin{array}
[c]{ll}%
\sin x, & \text{for }x\leq\pi/2,\\
2^{\sin x}, & \text{for }x>\pi/2,
\end{array}
\right.
\end{equation}
and let $y=\pi/2$. Then $g$ has a discontinuity at $y$. However, $g$
is zero-convex at $y$ with respect to $t=4/\pi$. Indeed, if $g(x)\leq0$, then
$x\leq0$. Therefore \eqref{eq:0subgradient} holds:%
\begin{equation}
g(y)+\langle t,x-y\rangle=1+4x/\pi-2<x\leq0.
\end{equation}

\end{example}

%%%%%%%%%%%%%%%%%%%%%%%%%%%%%%%%%%%%%%%%%%%%%%%%%%%%%%%%%%%%%%%%%%%%
%%%%%%%%%%%%%%%%%%%%%%%%%%%%%%%%%%%%%%%%%%%%%%%%%%%%%%%%%%%%%%%%%%%%
\begin{figure}
\begin{minipage}[t]{0.45\textwidth}
\begin{center}
{\includegraphics[clip,scale=0.8]{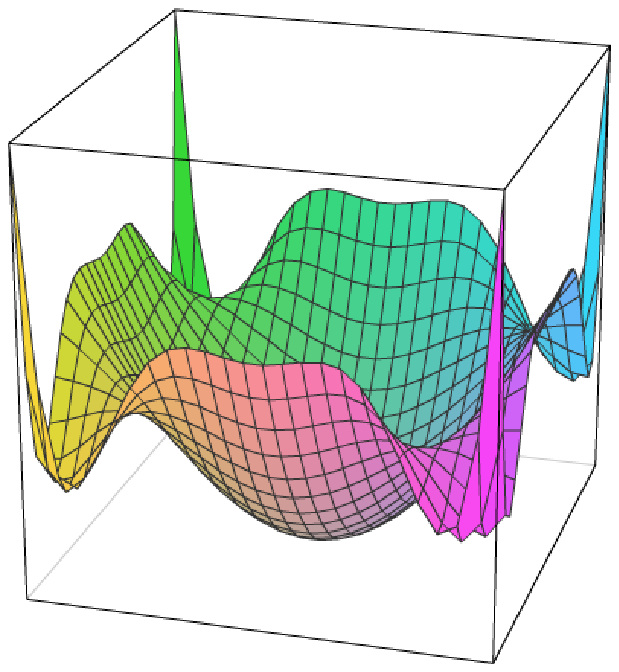}}
\end{center}
 \caption{An illustration of the polynomial of Example  \ref{ex:polynomial}.} 
\label{fig:0-ConvexPoly}
\end{minipage}
\begin{minipage}[t]{0.45\textwidth}
\begin{center}
{\includegraphics[clip,scale=0.83]{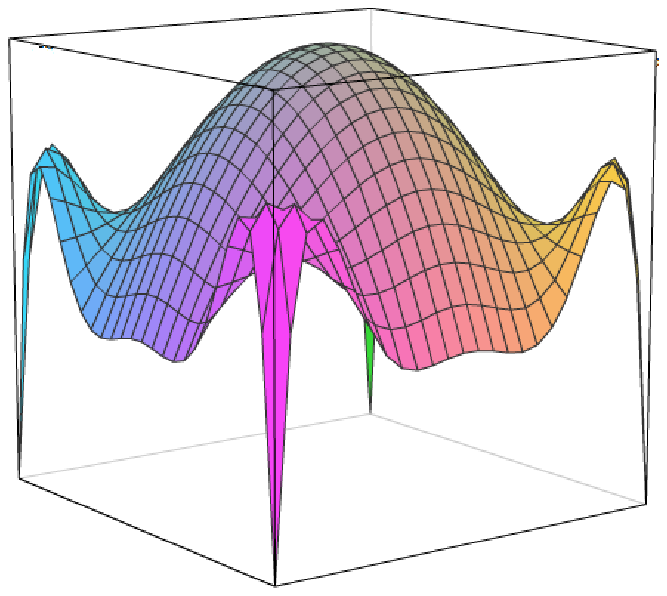}}
\end{center}
 \caption{Another illustration of the polynomial, now from the reverse perspective.} 
\label{fig:0-ConvexPolyReverse}
\end{minipage}
\end{figure}
%%%%%%%%%%%%%%%%%%%%%%%%%%%%%%%%%%%%%%%%%%%%%%%%%%%%%%%%%%%%%%%%%%%
%%%%%%%%%%%%%%%%%%%%%%%%%%%%%%%%%%%%%%%%%%%%%%%%%%%%%%%%%%%%%%%%%%%

\begin{example}
\label{ex:polynomial} Let $g:\mathbb{R}^{2}\rightarrow\mathbb{R}$ be defined
by%
\begin{equation}
g(x_{1},x_{2})=x_{1}^{2}+x_{2}^{2}-x_{1}^{4}x_{2}^{4}+x_{1}^{6}x_{2}%
^{6}/4-0.3.
\end{equation}
Elementary calculations (checking the principal minors and using polar
coordinates) show that $g$ is convex on the disk $D_{1}=\{(x_{1},x_{2}%
)\mid\,x_{1}^{2}+x_{2}^{2}<0.7^{2}\}$ and that $g^{\leq0}\subseteq
D_{2}=\{(x_{1},x_{2})\mid x_{1}^{2}+x_{2}^{2}\leq0.6^{2}\}$. In addition, it
is evident from Figures  \ref{fig:0-ConvexPoly} and \ref{fig:0-ConvexPolyReverse}
that $g$ is not quasiconvex. As for the computation of the 0-subgradients of
$g$, if $y=(y_{1},y_{2})\in D_{1}$, then we can simply take standard
subgradients, thus,%
\begin{equation}
t=\nabla g(y)=(2y_{1}-4y_{1}^{3}y_{2}^{4}+1.5y_{1}^{5}y_{2}^{6},2y_{2}%
-4y_{2}^{3}y_{1}^{4}+1.5y_{2}^{5}y_{1}^{6}).
\end{equation}
For $y\notin D_{1}$, we use  \eqref{eq:t_remark}. 
The line $M$ passing through the projection 
$m=(0.6y)/\Vert y\Vert$ of $y$ on $D_{2}$ and orthogonal to $y-m$ separates $y$
and $g^{\leq0}$. Thus from \eqref{eq:t_remark} we conclude that%
\begin{equation}
t=\frac{g(y)(y-m)}{\Vert y-m\Vert^{2}}=\frac{g(y)}{\Vert y\Vert(\Vert
y\Vert-0.6)}(y_{1},y_{2})
\end{equation}
is in $\partial^{0}g(y)$. As said in Section \ref{sec:Intro}, inequalities
involving nonconvex polynomials (sometimes of high degree) appear in
optimization problems
\cite{HenrLassCSM2004,HenrLass-Bookchap-2005,HenrLassTAC2012,HildebrandKoppe2013,LassSIOPT2001} 
and related fields such as economics and operations research \cite{ADSZ2010,CambiniMartein2009-book}. 
\end{example}

\begin{example}
\label{ex:Voronoi} Let $H$ be a real Hilbert space and $\Omega$ be a 
nonempty closed and convex subset of $H$. Let $p\in\Omega$ and $A\subseteq H$ be
given. Suppose that the distance $d(p,A)$ between $p$ and $A$ is positive.
Define a function $g:\Omega\rightarrow\mathbb{R}$ by%
\begin{equation}
g(x):=d(x,p)-d(x,A),\quad\forall x\in\Omega.\label{eq:Voronoi}%
\end{equation}
This function (or, actually, the so obtained family of functions)
is zero-convex. Indeed, as said in 
Remark \ref{rem:Geometric0-Convex} (second interpretation), it suffices to show that $g^{\leq 0}$ is closed and convex (it is nonempty because $p\in g^{\leq 0}$).   
Now, since $g^{\leq 0}=\{x\in H\mid d(x,p)\leq d(x,A)\}\bigcap \Omega$ and because $\Omega$ is 
closed and convex, it is sufficient to prove that the first set in the intersection is 
closed and convex. A computation shows that 
\begin{equation}\label{eq:Halfspace}
\{x\in H \mid d(x,p)\leq d(x,A)\}=\bigcap_{a\in A}\{x\in H \mid  d(x,p)\leq d(x,a)\}. 
\end{equation}
Since $p\neq a$ for each $a\in A$, each of the members in the above intersection 
is nothing but the closed half-space whose bounding hyperplane passes through $(p+a)/2$ and orthogonal 
to $p-a$.  Thus $\{x\in H\mid d(x,p)\leq d(x,A)\}$ is the intersection of closed and convex 
sets, and hence closed and convex.

The zero-level-set of this function $g$ is the, so-called, \texttt{Voronoi cell of } 
$p$ \texttt{ (restricted to }$\Omega
$\texttt{)} \texttt{ with respect to the set }$A$, and hence $g$ deserves the name 
``Voronoi function''. A particular  
and frequently explored case is where the set $A$ consists of finitely many distinct 
points $p_1, p_2,\ldots, p_{\ell}$. These points, together with the given point $p=p_0$, 
are called the \texttt{sites}, and the Voronoi cell corresponding to the site $p_i$ is the set 
$\{x\in \Omega\mid d(x,p_i)\leq d(x,p_j),\quad \forall j\neq i\}$. 
 The collection of these cells is the \texttt{Voronoi diagram induced by the sites}. 
 Voronoi diagrams have numerous applications in science and technology, see, e.g., 
 \cite{Aurenhammer,VoronoiWeb,OBSC}. As can be seen from these surveys, Voronoi diagrams have
applications also when the sites are assumed to have more general shapes than
points, such as lines segments, balls, and so on, and hence in this case the
set $A$ may be infinite. Traditionally, Voronoi diagrams have been
investigated in finite-dimensional spaces (especially in $\mathbb{R}^{2}$ and
$\mathbb{R}^{3}$), but recently they have been investigated in
infinite-dimensional spaces too
\cite{KopeckaReemReich,ReemISVD09,ReemGeometricStabilityArxiv}, and several
real-world and theoretical applications were mentioned there.
 
Returning to $g$, it can be shown, using the triangle inequality, that
$\left\vert g(x)\right\vert \leq\sup_{a\in A}\Vert p-a\Vert$ for every 
$x\in\Omega$ (in fact, because $\{p\}$ is a singleton, the right-hand side 
is equal to the Hausdorff distance between $\{p\}$ and $A$). Thus, when $A$ is bounded, 
then $g$ is bounded on $\Omega$. However, if in addition $\Omega=H$, then this 
implies that $g$ cannot be convex. Indeed, assume by way of negation that $g$ is convex. 
Then because it is 
proper (since it is finite) and  lower semicontinuous (actually continuous), 
it can be represented as the pointwise supremum of a nonempty family of 
continuous affine functions \cite[p. 91]{VanTiel1984}. Since $g$ is non-constant, 
at least one member $h$ in this family of affine functions must be non-constant. 
In other words, there exist $0\neq v\in H$ and $\alpha\in \R$ such that 
$h:=\langle v,\cdot\rangle+\alpha$ satisfies $h(x)\leq g(x)$ for all $x\in \Omega$. 
But $\lim_{t\to\infty} h(tv)=\infty$. Thus $g$ is not bounded, a contradiction 
to what was established before.

As a matter of fact, frequently $g$ is not even quasiconvex. Indeed, just consider
the simple case where $\Omega=H=\mathbb{R}^{2}$, $p=(0,0)$, $A=\{(0,1)\}$.
Then for $x=(-1,1),$ $z=(1,1)$, and $y=(0,1)$ we have $y\in\lbrack x,z]$ but
$g(x)=g(z)=\sqrt{2}-1<1=g(y)$. The same argument holds whenever $A$ contains
an isolated point and the dimension of the space is at least 2 and $\Omega=H$.
It can hold even if $A$ does not have any isolated point: just take
$p,x,y,z,\Omega,H$ as above but either $A=\{0\}\times\lbrack0.5,1]$ or
$A=\{0\}\times\lbrack0.5,\infty)$. However, in some symmetric configurations
$g$ may be quasiconvex: for instance, when $A$ is a sphere, $p$ is
the center of the corresponding ball, and $\Omega=H$.

Computation of the 0-subgradients of $g$ is possible by 
the description mentioned in Remark \ref{rem:0SubgradCompute} 
(especially equality \eqref{eq:t_remark}).  If $y\in g^{\leq0}$, 
 then obviously $0\in\partial^{0}g(y)$. Otherwise, the definitions of $g$ and $g^{\leq0}$ 
 imply that there exists an $a\in A$ such that $\Vert y-a\Vert<\Vert y-p\Vert$. If we denote by 
$M$ the bisector between $p$ and $a$, namely the set of all points in $H$
having equal distance to $p$ and to $a$, then $M$ is a hyperplane which is the
boundary of the half-space $\{x\in H\mid d(x,p)\leq d(x,a)\}$. The point $y$ is
located strictly inside the other half-space $\{x\in H\mid d(x,a)\leq d(x,p)\}$. 
Since $g^{\leq0}$ is contained in 
$\{x\in H\mid d(x,p)\leq d(x,a)\}$ (as explained in \eqref{eq:Halfspace} and above it) 
it follows that $M$ is a 
hyperplane separating $y$ and $g^{\leq0}$. Let $m$ be the orthogonal
projection of $y$ onto $M$. By \eqref{eq:t_remark} it follows that 
$t=g(y)(y-m)/\Vert y-m\Vert^{2}$ is in $\partial^{0}g(y)$. 

It is possible to represent $t$ in a more convenient way. 
 Indeed, note that the hyperplane 
$M$ defined above can be represented explicitly as 
$M=\{x\in H\mid\langle x-u_0, v\rangle=0\}$ where $u_0=0.5(a+p)$ 
and $v=(a-p)/\|a-p\|$. Since $m$ is the orthogonal
projection of $y$ onto $M$ we can write $y=m+\beta v$ where $\beta$ is some real number.  
This and the Pythagoras theorem imply the identity $\|y-u_0\|^2=\beta^2+\|y-\beta v-u_0\|^2$, 
from which it follows that $\beta=\langle y-u_0,v\rangle$. From \eqref{eq:t_remark} 
we conclude that 
\begin{equation}\label{eq:t_voronoi}
t=\frac{g(y)v}{\beta}=\frac{g(y)(a-p)}{\langle y-0.5(a+p),a-p\rangle}.
\end{equation}
This $t$ depends on $y$ but also on $a$. By an appropriate selection of $a\in A$ we can ensure that $\|t\|\leq 4$. In fact, we can even ensure that 
$\|t\|$ will be bounded above by a number arbitrarily close to 2 and sometimes 
even by 2 (when $d(y,A)$ is attained). Indeed, assume $y\notin g^{\leq 0}$. Let $\epsilon \in (0,0.5g(y))$ be arbitrary. 
Let $a\in A$ be chosen such that 
\begin{equation}\label{eq:a_select}
d(y,a)<d(y,A)+\epsilon. 
\end{equation}
From the definition of $g$ and the triangle inequality we see that  
any point $x$ in the open ball of radius $0.5g(y)-\epsilon$ around $y$ satisfies 
\begin{multline}
d(x,a)\leq d(x,y)+d(y,a)<0.5g(y)-\epsilon+d(y,A)+\epsilon\\
\leq d(y,p)-g(y)+0.5g(y)\leq d(y,x)+d(x,p)-0.5g(y)<d(x,p),
\end{multline}
and hence $x$ belongs to the half-space to which $y$ belongs. Recalling that $|\beta|=\|y-m\|$ 
and that $m$ is in the other half-space, we have  
$|\beta|\geq 0.5g(y)-\epsilon$. This and \eqref{eq:t_voronoi} show 
that  
\begin{equation}\label{eq:t_norm}
\|t\|\leq\frac{g(y)}{0.5g(y)-\epsilon}. 
\end{equation}
This proves the claim since $\epsilon$ can be arbitrary small and 
we can select the appropriate $a\in A$ as above so that \eqref{eq:a_select} and 
hence \eqref{eq:t_norm} will be satisfied. 
In particular, by taking $\epsilon=0.25g(y)$ we obtain $\|t\|\leq 4$. If in addition  $d(y,A)=d(y,a)$ for some $a\in A$, then 
by choosing this $a$ and mimicking the previous analysis with $\epsilon=0$ we see  
that $\|t\|\leq 2$. 
\end{example}

\begin{example}\label{ex:VoronoiWeighted}
The functions described below are variations of the Voronoi function defined in 
\eqref{eq:Voronoi}. They deserve some attention since a particular case 
of them will be used in Section \ref{sec:ComputationalResults}. 

One variation is obtained by replacing $p$ by a subset $P$ and taking  
\begin{equation}
g(x):=g_{P,A}(x)=d(x,P)-d(x,A). \label{eq:P}%
\end{equation}
See  \cite{Aurenhammer,VoronoiWeb,KopeckaReemReich,OBSC,ReemGeometricStabilityArxiv} 
and the references therein for some applications of Voronoi cells defined in 
this way. In general, the Voronoi cell $g^{\leq0}$ is closed ($g$ is 2-Lipschitz) but not convex. However, 
 in some cases it is convex, e.g., when $\Omega=H=\mathbb{R}^{2}$, 
$A=\{(-1,0),(0,-1),(1,-1),(0,1),(1,1),(2,0)\}$, and $P=\{(0,0),(1,0)\}$. 

Another variation is to consider weighted distances, namely, we assign to each $a\in A$ a real 
number $w_{a}$ (a weight), and assign a weight $w_{p}$ to $p$. 
 For every $x\in \Omega$ and $a\in A$ let 
\begin{equation}%
\begin{array}
[c]{lll}%
d_{p}(x) & := & \Vert x-p\Vert-w_{p},\\
d_{a}(x) & := & \Vert x-a\Vert-w_{a},\\
d_{A}(x) & := & \inf\{d_{a}(x)\mid a\in A\},
\end{array}
\end{equation}
and define the additively weighted  Voronoi function 
\begin{equation}
g(x):=g_{w}(x):=d_{p}(x)-d_{A}(x). \label{eq:weighted}%
\end{equation}
The 0-level-set $g^{\leq 0}$ is the additively weighted Voronoi cell of the site $p$. 
It is closed since $d_A$ is upper semicontinuous and 
hence $g$ is lower semicontinuous. In molecular biology 
\cite{GPF1997,KWCKLBK2006,RKCPK2005} the site $p$ 
represents the center of a spherical atom (or molecule) whose van der Waals radius is $w_p$. 
Hence $d_p(x)$ is the distance from $x$ to the sphere 
for each $x$ outside the corresponding ball. Similarly, each $a\in A$ represents 
the center of a spherical atom (or molecule) whose van der Waals  
radius is $w_a$. In crystallography and stochastic geometry 
the common name to additively weighted 
Voronoi diagrams is Johnson-Mehl tessellation (or model). In this model 
$p$ (and each $a\in A$) represent a nucleation center from which a crystal starts 
to grow in a uniform way in all directions, but the growing process starts at 
different times from each nucleation center. In this case $w_p$ is minus the 
starting time of the growth from $p$ and $w_a$ is minus the 
starting time of the growth from $a$. See, e.g., \cite{CSKM2013,OBSC} and the 
references therein. See also Section \ref{sec:ComputationalResults} below for 
a concrete computational result in the molecular biology context.

Under certain assumptions on the parameters the function $g$ defined in \eqref{eq:weighted} 
is zero-convex. 
For instance, assume that $A=\{a\}$, $a\in H$ is given, and that 
\begin{equation}\label{eq:wp_wa}
w_p\leq w_a<\|a-p\|+w_p. 
\end{equation}
Let $B$ be the ball of 
radius $w_a-w_p$ around $a$ (degenerates to a point when $w_a=w_p$). 
We claim that under these assumptions 
\begin{equation}\label{eq:gh}
g^{\leq 0}=G^{\leq 0}  
\end{equation}
where $G(x)=d(x,p)-d(x,B)$ for all $x\in \Omega$. 
Indeed, if $x\notin B$, then we have the equality $G(x)=d(x,p)-(d(x,a)-(w_a-w_p))=g(x)$.  Hence the intersection 
of both sides of \eqref{eq:gh} with the complement of $B$ coincide. 
If $x\in B$, then $G(x)=d(x,p)$ and hence 
$x\in G^{\leq 0}$ would imply that $x=p$, a contradiction to $p\notin B$ 
(by \eqref{eq:wp_wa}). Therefore $x\notin G^{\leq 0}$. 
However, the assumption $x\in B$ implies 
$d(x,a)\leq w_a-w_p$. Hence $x$ cannot belong to $g^{\leq 0}$ because 
this would imply that $d(x,p)\leq d(x,a)-(w_a-w_p)\leq 0$, and again $x=p$, 
a contradiction. We conclude that \eqref{eq:gh} holds. Since we already know 
from Example \ref{ex:Voronoi} that $G^{\leq 0}$ is convex, 
it follows that $g^{\leq 0}$ is convex. Since $g^{\leq 0}$ is closed, 
Remark \ref{rem:Geometric0-Convex} (second interpretation) implies that $g$ is zero-convex. 
Geometrically (at least when $\Omega$ is, say, the whole space or it is a cube 
containing $p$ and $a$ in its interior), the boundary of 
$g^{\leq 0}$ is the intersection of 
$\Omega$ with the (possibly infinite dimensional) hyperboloid  
$\{x\in H \,| \,d(x,p)-d(x,a)=w_p-w_a\}$. In addition, $p\in g^{\leq 0}$. 
When $w_p=w_a$, the hyperboloid degenerates to a hyperplane.  

We finish this example by noting that there are other weighted 
versions of Voronoi diagrams. One of them is the multiplicative weighted distance in which 
$d_p(x)=d(x,p)/w_p$ and $d_a(x)=d(x,a)/w_a$ for some given positive 
weights $w_p$ and $w_a$. This version is used in molecular biology, 
e.g., in \cite{GTL1995}, where again $w_p$ and $w_a$ are the van der Waals 
radii of the involved atoms/molecules. See also \cite{Aurenhammer,OBSC}. 
\end{example}

\section{Zero-convex functions: Properties}

\label{sec:properties}

In this section we present several properties of zero-convex functions and 
discuss theoretical ways of constructing their 0-subgradients.

\begin{proposition}
\label{prop:0ConvCharacterize} Let $H$ be a real Hilbert space and let
$\Omega$ be a nonempty convex subset of $H$. Let $g:\Omega\rightarrow\mathbb{R}$  
be given. 

\begin{enumerate}
[(a)]

\item \label{item:PsiConvex} The function $g$ is
zero-convex at $y\in\Omega$ with respect to some $t\in \partial^0g(y)$ if and only if there exists a function
$\psi:\mathbb{R}\rightarrow\mathbb{R}$ satisfying $\psi(r)\leq0$ for all
$r\leq0$ such that%
\begin{equation}
g(y)+\langle t,x-y\rangle\leq\psi(g(x)),\quad\forall x\in g^{\leq0}.
\label{eq:PsiSubgradient}%
\end{equation}

\item \label{item:0ConvexImplyConvexZ} If $g$ is
zero-convex, then its zero-level-set $g^{\leq0}$ is convex.

\item \label{item:LevelSet} If $g^{\leq0}$ is closed and convex, then $g$ is 
zero-convex. In fact, if $y\in g^{\leq0}$, then $0\in\partial^{0}g(y)$, 
and if $y\notin g^{\leq0}$, then for 
\begin{equation}\label{eq:t}
t=\frac{g(y)}{\Vert y-m\Vert^{2}}(y-m), 
\end{equation}
we have $t\in\partial^{0}g(y)$ where $m\in M$ is the orthogonal projection of 
$y$ onto  a (closed) hyperplane $M$ strictly 
separating $y$ from $g^{\leq0}$. 
\item\label{item:tm} If $m$ is the (unique) orthogonal projection of 
$y\notin g^{\leq0}$ onto $g^{\leq0}$, then for $t$ defined in \eqref{eq:t} 
we have $t\in\partial^{0}g(y)$. 
\end{enumerate}
\end{proposition}

\begin{proof}
It can be assumed that $g^{\leq 0}\neq\emptyset$, otherwise 
the assertion holds trivially (void). 
\begin{enumerate}[(a)]

\item If there exists such a function $\psi$, then \eqref{eq:PsiSubgradient}
implies \eqref{eq:0subgradient} since $g(x)\leq0$ implies $\psi(g(x))\leq0$.
Hence $g$ is zero-convex at $y$ and $t\in \partial^0 g(y)$. Conversely, if $g$ is zero-convex 
at $y$ and $t\in \partial^0 g(y)$, then \eqref{eq:PsiSubgradient} is satisfied with any function $\psi
:\mathbb{R}\rightarrow\mathbb{R}$ satisfying $\psi(r)=0$ whenever $r\leq0$.

\item Suppose, by way of negation, that $g^{\leq0}$ is not convex. Then there
exist two distinct points $x_{1},$ $x_{2}\in g^{\leq0}$ such that for some $y$
in the line segment $[x_1,x_2]$ we have $y\notin g^{\leq0}$, namely
$g(y)>0$. Since $g$ is zero-convex on the convex subset $\Omega,$ thus at $y,$ there is 
a point $t\in H$ such that \eqref{eq:0subgradient} holds. This and the fact
that $g(x_{i})\leq0$, $i=1,2,$ imply that the function%
\begin{equation}
f(x):=g(y)+\langle t,x-y\rangle
\end{equation}
satisfies $f(x_{i})\leq0$, $i=1,2$. Since $f(x)$ is convex and $y\in\lbrack
x_{1},x_{2}]$ we also have $f(y)\leq0$. This is a contradiction since
$f(y)=g(y)>0$.

\item   Given $y\in\Omega$, distinguish between the cases $y\in g^{\leq0}$ or
$y\notin g^{\leq0}$. In the first case define $t:=0$. Then for any $x\in
g^{\leq0}$ we obviously have%
\begin{equation}
g(y)+\langle t,x-y\rangle\leq0+0, \label{eq:0+0}%
\end{equation}
hence, \eqref{eq:0subgradient} is satisfied. Now consider the case $y\notin
g^{\leq0}$. Since $g^{\leq0}$ is closed and convex, the Hahn-Banach theorem,
in one of its geometric versions \cite[p. 38]{VanTiel1984}, ensures that there exists a 
 hyperplane $M$ strictly separating $y$ from $g^{\leq0}$. The hyperplane $M$ is 
  guaranteed to be a closed set  and, actually, it can be written as $M=\{x\in H\mid\langle e,x-m\rangle
=0\}$ where $m\in M$ is the orthogonal projection of $y$ onto $M$ and
$e=(y-m)/\Vert y-m\Vert$. We have the decomposition $H=M\cup H_{1}\cup H_{2}$
where $H_{1}=\{x\in H\mid\langle e,x-m\rangle>0\}$ and $H_{2}=\{x\in
H\mid\langle e,x-m\rangle<0\}$. By the definition of $m$, $M$, and $e$ it
follows that $y\in H_{1}$ and $g^{\leq0}\subseteq M\cup H_{2}$. Let
$\beta:=g(y)/\Vert y-m\Vert$ and $t:=\beta e$, as in \eqref{eq:t}. Since
$g(y)>0$ we have $\beta>0$. The above implies that for each $x\in g^{\leq0}$%
\begin{equation}
g(y)+\langle t,x-y\rangle=\langle t,y-m\rangle+\langle t,x-y\rangle
=\beta\langle e,x-m\rangle\leq0, \label{eq:beta}%
\end{equation}
thus, \eqref{eq:0subgradient} is satisfied again. As a matter of fact, by
translating $M$ slightly towards $y$ we can even ensure that $g^{\leq0}\subset
H_{2}$ and so \eqref{eq:beta} will be satisfied with strict inequality.

\item Because $m$ is the  orthogonal projection of 
$y\notin g^{\leq0}$ onto $g^{\leq0}$ (whose existence and uniqueness are well-known, 
see, e.g., \cite{GG-book-1981}), then the hyperplane $M$ 
which passes through $m$ and is orthogonal to $y-m$ strictly separates $y$ and $g^{\leq0}$. 
Indeed, since $\langle y-m,y-m\rangle>0$ (otherwise $y\in g^{\leq 0}$) we have 
$y\in H_{1}=\{x\in H\mid\langle y-m,x-m\rangle>0\}$. We can write 
$M=\{x\in H\mid\langle y-m,x-m\rangle=0\}$ and we have the decomposition $H=M\cup H_{1}\cup H_{2}$ where 
 $H_{2}=\{x\in H\mid\langle y-m,x-m\rangle<0\}$. A well-known characterization of the 
 orthogonal projection $m$ of a point $y$ onto a nonempty, closed and convex 
subset says that $\langle y-m,x-m\rangle \leq 0$ 
for every $x$ in the subset (see \cite[p. 46]{BauschkeCombettes2011}). Therefore $g^{\leq 0}\subseteq M\cup H_2$ 
and hence $g^{\leq 0}\cap H_1=\emptyset$. Thus $M$ strictly separates 
$y$ and $g^{\leq 0}$ and the assertion follows from part \eqref{item:LevelSet}. 
\end{enumerate}
\end{proof}

\begin{remark}
\label{rem:ConevxZeroLevelImplyConvex} An alternative but somewhat related
approach to the fact that the
convexity of $g^{\leq0}$ implies the zero-convexity of $g$, based on an idea
of Benar Svaiter \cite{SvaiterPC2011}, is as follows. 
Assume that $g^{\leq 0}\neq\emptyset$, otherwise the assertion holds trivially (void). 
Define  the distance from $x$ to $g^{\leq0}$ as $f(x):=d(x,g^{\leq0})$ for each $x$. 
 This continuous function is also convex since
$g^{\leq0}$ is convex. Hence, as is well-known \cite[p.76]{VanTiel1984}, it has a (standard) subgradient $s$ at any
$y$. Let $t:=cs$ where%
\begin{equation}
c:=\left\{
\begin{array}
[c]{ll}%
g(y)/f(y), & \text{if }f(y)\neq0,\\
0, & \text{otherwise.}%
\end{array}
\right.
\end{equation}
Note that $f(y)=0$ if and only if $y\in g^{\leq0}$ because $g^{\leq0}$ is
closed. This implies \eqref{eq:0subgradient} when $y\in g^{\leq0}$ with $t=0$
(as in \eqref{eq:0+0}). When $y\notin g^{\leq0}$ we have $f(y)>0$, $g(y)>0$,
and $t=(g(y)/f(y))s$. By the subgradient inequality, which $f$ satisfies, we
have%
\begin{equation}
f(y)+\langle s,x-y\rangle\leq f(x)=0,\quad\forall x\in g^{\leq0}.
\end{equation}
This implies \eqref{eq:0subgradient} after multiplying this inequality by
$c=g(y)/f(y)$.
\end{remark}

\begin{corollary}
\label{cor:Quasiconvex} Let $H$ be a Hilbert space and $\Omega\subseteq H$ be nonempty,  
closed, and convex. Any lower semicontinuous function $g:\Omega\to\R$ having a convex 
zero-level-set is zero-convex. In particular, if $g$ is lower semicontinuous and 
quasiconvex, then it is zero-convex.
\end{corollary}

\begin{proof}
It can be assumed that $g^{\leq 0}\neq\emptyset$, otherwise 
the assertion holds trivially (void). 
Since $g$ is lower semicontinuous $g^{\leq0}$ is closed in  $\Omega$ (in the topology 
induced by the norm) and hence ($\Omega$ is closed) in $H$. Thus, when $g^{\leq0}$ is 
assumed to be convex the assertion follows 
from Proposition \ref{prop:0ConvCharacterize}\eqref{item:LevelSet}. The
assertion about lower semicontinuous quasiconvex functions is a consequence of
Proposition  \ref{prop:0ConvCharacterize}\eqref{item:LevelSet} and the fact
that all their level-sets are closed and convex.
\end{proof}

\begin{proposition}
\label{prop:0Properties} Let $H$ be a real Hilbert space and let $\Omega$ be a 
nonempty convex subset of $H$.

\begin{enumerate}
[(a)]

\item \label{item:ConvComb} If $g:\Omega\rightarrow\mathbb{R}$ is zero-convex
at $y$ with respect to both $t_{1}\in\partial^{0}g(y)$ and $t_{2}\in
\partial^{0}g(y)$, then it is zero-convex at $y$ with respect to any convex
combination of $t_{1}$ and $t_{2}$.

\item \label{item:Alpha} Suppose that $g:\Omega\rightarrow\mathbb{R}$ is
zero-convex at $y$ with respect to some $t\in\partial^{0}g(y)$. Given
$\alpha\geq0$, the function $\tilde{g}:=\alpha g$ is zero-convex at $y$ with
respect to $\tilde{t}=\alpha t$.

\item \label{ex:Max} Suppose that $g_{1},g_{2},\ldots,g_{m}$ are given
zero-convex functions at $y$. Then the envelope of $\{g_{i}\}_{i=1}^{m}$,
defined by $g(x):=\max\{g_{i}(x)\mid\,i=1,2,\ldots,m\}$, is also zero-convex
at $y$.

\item \label{item:sup} Suppose that $\{g_{i}|\,\,i\in I\}$ is a family of
(possibly infinitely many) lower semicontinuous zero-convex functions defined 
on a closed subset $\Omega$ of $H$. Then $g=\sup\{g_{i}\mid i\in I\}$ is zero-convex.

\item \label{item:Composition} Suppose that $g:\Omega\rightarrow\mathbb{R}$ is 
zero-convex and that it has a closed zero-level-set. 
Let $\psi:\mathbb{R}\rightarrow\mathbb{R}$  be a function satisfying 
$\psi(r)\leq0$ if and only if $r\leq0$. Then the composite function 
$\psi\circ g$ is zero-convex. In particular the above holds when $\Omega$ 
is closed, $g$ is lower semicontinuous and zero-convex, and $\psi$  satisfies 
the above-mentioned property. 

\item \label{item:NonVanishSubgrad} Suppose that $g:\Omega\rightarrow
\mathbb{R}$ has a nonempty zero-level-set. If $g$ is zero-convex at $y$ and if
$g(y)>0$, then any 0-subgradient $t\in\partial^{0}g(y)$ satisfies $t\neq0$.

\item \label{item:fg} Suppose that $f:\Omega\rightarrow\mathbb{R}$ and
$g:\Omega\rightarrow\mathbb{R}$ are zero-convex at $y\in\Omega$ and that their
zero-level-sets coincide. If $y$ is outside the zero-level-set and
$t\in\partial^{0}f(y)$, then $ct\in\partial^{0}g(y)$ for any $c\geq
g(y)/f(y)$. In particular, if $t\in\partial^{0}g(y)$, then so does $ct$ for all 
$c\geq1$.
\end{enumerate}
\end{proposition}

\begin{proof}
\begin{enumerate}
[(a)]

\item Follows from multiplication of \eqref{eq:0subgradient} by each of the
convex combination coefficients and adding the resulting inequalities.

\item Follows from multiplication of \eqref{eq:0subgradient} by $\alpha$.

\item For each $i$ consider the associated subgradient $t_{i}\in\partial
^{0}g_{i}(y)$. Let $j$ be the index for which $g(y)=g_{j}(y)$ and let
$t=t_{j}$. Suppose that $x\in H$ satisfies $g(x)\leq0$. Then $g_{j}(x)\leq0$,
thus, by \eqref{eq:0subgradient},%
\begin{equation}
\langle t,x-y\rangle+g(y)=\langle t_{j},x-y\rangle+g_{j}(y)\leq0,
\end{equation}
as required.

\item Let $x_{1},x_{2}\in g^{\leq0}$. Then $g_{i}(x_{1})\leq g(x_{1})\leq0$
and $g_{i}(x_{2})\leq0$ for any $i\in I$. From  
the convexity of the zero-level-set of $g_{i}$ (Proposition 
\ref{prop:0ConvCharacterize}\eqref{item:0ConvexImplyConvexZ}) it follows that 
$g_{i}(x)\leq0$ for all $x$ in the line segment $[x_{1},x_{2}]$  and all $i$. Thus $g(x)\leq0$ and hence $g^{\leq0}$ is convex.
It is well-known and not hard to verify that $g$ is lower semicontinuous. Therefore,  
Corollary \ref{cor:Quasiconvex} implies that $g$ is zero-convex. 

\item By assumption $g^{\leq 0}$ is closed and by 
Proposition \ref{prop:0ConvCharacterize}\eqref{item:0ConvexImplyConvexZ} it 
is convex. By the nature of $\psi$ the zero-level-sets of $g$ and of $\psi\circ g$ coincide. 
Thus Proposition \ref{prop:0ConvCharacterize}\eqref{item:LevelSet} implies that 
$\psi\circ g$ is zero-convex. Finally, if $\Omega$ is closed and $g$ is also 
lower semicontinuous, then $g^{\leq 0}$ is closed and the assertion follows from the above discussion. 

\item Let $t\in\partial^{0}g_{i}(y)$ and assume, to the contrary, that $t=0$.
If $x\in\Omega$ satisfies $g(x)\leq0$, then by \eqref{eq:0subgradient}%
\begin{equation}
0<g(y)=g(y)+\langle t,x-y\rangle\leq0,
\end{equation}
which is a contradiction.

\item From \eqref{eq:0subgradient} and the equality between the zero-level-sets
of the functions we have the inequality $\langle t,x-y\rangle\leq-f(y)$ for any $x\in
g^{\leq0}$. In addition, $f(y)>0$, thus,%
\begin{equation}
g(y)+\langle ct,x-y\rangle\leq g(y)-cf(y)\leq0,
\end{equation}
by the choice of $c$. Therefore, $ct\in\partial^{0}g(y)$. Finally, by taking
$f=g$ in the previous case, we conclude that if $t\in\partial^{0}g(y)$, then
$ct\in\partial^{0}g(y)$ for any $c\geq1$.
\end{enumerate}
\end{proof}

For later use  (see, e.g., the discussion after Condition
\ref{cond:boundedness} below) we present a few propositions which give sufficient
conditions for the existence of bounded 0-subgradients. These propositions also
give some ideas regarding the way of computing 0-subgradients in certain
settings. The first proposition is a generalized variation of an assertion
hidden in the proof of \cite[Proposition 7.8, (ii)$\Longrightarrow$(iii)]{bb96}
(namely, that the subgradients of a convex and Lipschitz function are
uniformly bounded by the Lipschitz constant).

\begin{proposition}
\label{prop:alpha} Let $H$ be a real Hilbert space and let $\Omega$ be nonempty, closed and convex. 
Suppose that $g:\Omega\rightarrow\mathbb{R}$ is zero-convex
and Lipschitz on $\Omega$ with a Lipschitz constant $\alpha>0$ and that $g^{\leq 0}\neq\emptyset$.  
Then for each $y\in\Omega$ there exists $t\in\partial^{0}g(y)$ satisfying $\Vert t\Vert
\leq\alpha$. As a matter of fact, for each $y\notin g^{\leq0}$ there exists
$t\in\partial^{0}g(y)$, defined by \eqref{eq:t}, satisfying%
\begin{equation}
\Vert t\Vert\leq\alpha\frac{\Vert y-z\Vert}{\Vert y-m\Vert},
\label{eq:alpha_z_m}%
\end{equation}
where $m\in M$ is the projection of $y$ onto a closed hyperplane strictly
separating $y$ from $g^{\leq0}$, and $z\in g^{\leq0}$ is arbitrary. In
particular, the above holds if $\Omega$ is contained in an open subset of $H$ 
and $g$ is G\^{a}teaux-differentiable on this open set 
and its derivative is uniformly bounded by some $\alpha>0$.
\end{proposition}

\begin{proof}
If $y\in g^{\leq0}$, then we can take $t=0$. Otherwise, let $t\in\partial
^{0}g(y)$ be defined as in \eqref{eq:t} where $m\in M$ is the projection of
$y$ onto a closed hyperplane $M$ strictly separating $y$ from $g^{\leq0}$, the
existence of which is ensured by the Hahn-Banach theorem since $g^{\leq0}$ is
closed ($g$ is continuous and $\Omega$ is closed) and convex (Proposition
\ref{prop:0ConvCharacterize}\eqref{item:0ConvexImplyConvexZ}). From
\eqref{eq:t} we have $\Vert t\Vert=g(y)/\Vert y-m\Vert$. Let $z\in g^{\leq0}$
be given. Since $g(z)\leq0$ and $g(y)>0$, the fact that the Lipschitz
constant of $g$ is $\alpha$ implies that%
\begin{equation}
\Vert t\Vert\leq\frac{(g(y)-g(z))\Vert y-z\Vert}{\Vert y-z\Vert\Vert y-m\Vert
}\leq\alpha\frac{\Vert y-z\Vert}{\Vert y-m\Vert}.
\end{equation}
In particular, the above is true when $m$ is the best approximation
(orthogonal) projection of $y$ onto $g^{\leq0}$ and $M$ passes through $m$ and
is orthogonal to $y-m$ (see the proof of Proposition \ref{prop:0ConvCharacterize}\eqref{item:tm}). 
By taking $z=m$ we have $\Vert t\Vert\leq\alpha$, as claimed. 
Finally, a well-known consequence of the mean value theorem says that when $g$ 
is G\^{a}teaux-differentiable and its
derivative is bounded by some constant, then $g$  is Lipschitz with
this constant \cite[Theorem 1.8, p. 13]{AmbroProdi-Book-1993} and hence the assertion follows.
\end{proof}

\begin{proposition}
\label{prop:EpsilonConvex}  Let $H$ be a real Hilbert space and let $\Omega$ be nonempty and convex.
 Let $g:\Omega\rightarrow\mathbb{R}$ be zero-convex.
Suppose that $\emptyset \neq g^{\leq0}\subseteq B(c,r)$ 
where $B(c,r)$ is the open ball with center $c$ and radius $r>0$. 
Let $\epsilon>0$ be given. Suppose that $B(c,r+\epsilon)\subset \Omega$ and that $g$ is convex on 
this ball. 

\begin{enumerate}
[(a)]

\item If $g$ is Lipschitz on $\Omega$ with constant $\alpha$, then
for each $y\in\Omega$ there exists $t\in\partial^{0} g(y)$ satisfying
$\|t\|\leq\alpha(1+(2r/\epsilon))$. In fact, if $y\in B(c,r+\epsilon)$, then $t$ can be taken 
as a standard subgradient and if $y\notin B(c,r+\epsilon)$, then $t$ 
can be defined by \eqref{eq:t}.

\item If $g$ is bounded on $\Omega$ by some $\beta>0$, then for each
$y\in\Omega$ there exists $t\in\partial^{0}g(y)$ satisfying $\Vert t\Vert
\leq 4\beta/\epsilon$. In fact, if $y\in B(c,r+0.5\epsilon)$,
then $t$ can be taken as a standard subgradient, and if $y\notin B(c,r+0.5\epsilon)$, then 
this $t$ can be defined by \eqref{eq:t}.
\end{enumerate}
\end{proposition}

\begin{proof}
\begin{enumerate}
[(a)]

\item Since $g$ is Lipschitz and convex on the ball $B(c,r+\epsilon)$ 
with a Lipschitz constant $\alpha$, it is known that its standard subgradients 
at points in this ball are bounded by $\alpha$ (see the proof of  \cite[Proposition 7.8, (ii)$\Longrightarrow$
(iii)]{bb96} and replace there $rB_{X}$ by our ball). Any standard subgradient 
is a 0-subgradient as explained in Example \ref{ex:convex} above. Now let $y\in\Omega$,
$y\notin B(c,r+\epsilon)$ and consider its projection $m$ onto the closed 
ball $\overline{B(c,r)}$. Consider also the closed hyperplane $M$ passing through $m$ and orthogonal to $y-m$. 
This $M$ separates the ball and hence $g^{\leq0}$ from $y$  and for $t$ defined by
\eqref{eq:t} we know from 
Proposition \ref{prop:0ConvCharacterize}\eqref{item:LevelSet} that $t\in\partial^{0}g(y)$. Let $z\in g^{\leq0}$ be
given. From \eqref{eq:alpha_z_m}, the fact that $z,m\in\overline{B(c,r)}$, and
the fact that $\Vert y-m\Vert\geq\epsilon$, we have%
\begin{equation}
\Vert t\Vert\leq\alpha\frac{\Vert y-z\Vert}{\Vert y-m\Vert}\leq\alpha
\frac{(\Vert y-m\Vert+\Vert m-z\Vert)}{\Vert y-m\Vert}\leq\alpha\left(
1+\frac{2r}{\epsilon}\right). \label{eq:2r_epsilon}%
\end{equation}
Since the right-hand side of \eqref{eq:2r_epsilon} is greater than $\alpha$ we
conclude that in both cases discussed above we can find $t\in\partial^{0}g(y)$
satisfying $\Vert t\Vert\leq\alpha(1+(2r/\epsilon))$.

\item The restriction of $g$ to the ball $B(c,r+\epsilon)$ is a convex
function which is bounded by $\beta$, thus it is a known fact 
(which follows, e.g., from the proof of \cite[Theorem 5.21, p. 69]{VanTiel1984}) that 
$g$ is Lipschitz on the ball $B(c,r+0.5\epsilon)$ with  constant 
$2\beta/(r+0.5\epsilon-r)=4\beta/\epsilon$. Hence, as explained before, any standard subgradient (which is 
a 0-subgradient) of a point $y$ in the ball is bounded by this Lipschitz constant. 
Now consider a point $y$ outside this ball. For $t$ defined by \eqref{eq:t} we know 
from Proposition \ref{prop:0ConvCharacterize}\eqref{item:LevelSet} that $t\in 
\partial^{0}g(y)$ and $\Vert t\Vert=g(y)/\Vert y-m\Vert\leq 2\beta/\epsilon$.
The assertion follows. 
\end{enumerate}
\end{proof}

\begin{remark}\label{rem:Unbounded0Subgrad}
In general, if some of the above conditions are not satisfied, then uniform
boundedness of a selection of 0-subgradients cannot be ensured even if the given function
is continuous and quasiconvex. A simple example is 
$g:\mathbb{R}\to\mathbb{R}$ defined by $g(y)=0$ when $y\leq0$, and
$g(y)=\sqrt{y}$ otherwise. This is a continuous and quasiconvex function, but 
when $y>0$ and $t\in\partial^{0} g(y)$, it follows from
\eqref{eq:0subgradient} (by putting $x=0\in g^{\leq0}$) that $t\geq
1/\sqrt{y}$.
\end{remark}

\section{Formulation of the zero-convex feasibility problem and the associated
algorithm\label{sec:alg}}
In this section we formulate our algorithm for solving the CFP with zero-convex functions. 
See Section \ref{sec:ComputationalResults} below for a concrete example (including computational 
results). See also 
Subsection \ref{subsec:ZeroConvex} and Section \ref{sec:zero-convex} above for related examples. 

Let $\Omega$ be 
a nonempty closed and convex subset of the real Hilbert space $H$.
Denote by $P_{\Omega}$ the best approximation (orthogonal) projection onto $\Omega$. 
Let $J$ be a finite or a 
countable set of indices. For each $j\in J$, let $g_{j}:\Omega\rightarrow
\mathbb{R}$ be a continuous zero-convex function.  For each $j\in J$ let%
\begin{equation}
C_{j}=\{x\in\Omega\mid g_{j}(x)\leq0\}\label{eq:C_j}%
\end{equation}
and suppose that
\begin{equation}
C=\bigcap_{j\in J}C_{j}\neq\emptyset.\label{eq:nonempty}%
\end{equation}
Let $\{i(n)\}_{n=0}^{\infty}$ be an infinite sequence indices $i(n)\in J$,
henceforth called a \textit{control sequence}, which is  \textit{almost
cyclic in a generalized sense}, i.e., $i:\mathbb{N}\cup \{0\}\rightarrow J$ 
and for each $j\in J$ there exists an $L_{j}\in\mathbb{N}$ such that the control selects
the subset $C_{j}$ at least once in each block of length $L_{j}$ of successive
indices of $J$. Formally,
\begin{equation}
\forall j\in J\text{, }\exists L_{j}\in\mathbb{N}\,\text{\ such that }%
\forall\,\,s\in\mathbb{N}\,\ \text{we have }j\in\{i(s),i(s+1),\ldots
,i(s+L_{j}-1)\}.\label{eq:control}%
\end{equation}
This definition seems to have been introduced by Browder \cite[Definition
5]{Browder1967}. See \cite[pp. 209--210]{Combettes1996} for an 
example with $L_{j}=2^{j},j\in J=\mathbb{N}$. A well-known particular case of  \eqref{eq:control} is the 
almost cyclic control, namely $J=\{1,2,\ldots,\ell\}$, $\ell\in \N$ is given, 
and there exists $L\in\N$ such that $L_j=L$ for all $j\in J$. 
The particular case of the almost cyclic control when $L=\ell$ is the 
cyclic control. For other types of controls which are 
related to \eqref{eq:control}, see \cite{CensorChenPajoohesh2011}. 

We consider the following algorithm.

\begin{alg}
\label{alg:perturbed-csp}$\,$\textbf{The Sequential Subgradient Projection 
(SSP) Method with Perturbations}\newline

{\noindent\textbf{Initialization:}} $x_{0}\in\Omega$ is arbitrary.\newline

{\noindent\textbf{Iterative Step:}}%
\begin{equation}
x_{n+1}=\left\{
\begin{array}
[c]{ll}%
P_{\Omega}\left(  x_{n}-\displaystyle\lambda_{n}{\frac{g_{i(n)}(x_{n}%
)}{\parallel t_{n}\parallel^{2}}}t_{n}+b_{n}\right)  , & \text{if }%
g_{i(n)}(x_{n})>0,\\
x_{n}, & \text{if }g_{i(n)}(x_{n})\leq0,
\end{array}
\right.  \label{eq:iterative}%
\end{equation}
where $t_{n}\in\partial^{0}g_{i(n)}(x_{n})$ and $\{b_{n}\}_{n=0}^{\infty}$ is a
sequence of elements in $H$.\newline

{\noindent\textbf{Relaxation Parameters: }}$\{\lambda_{n}\}_{n=0}^{\infty}$ is a
sequence of real numbers satisfying the inequality 
\begin{equation}\label{eq:relaxation}
\epsilon_{1}\leq \lambda_{n}\leq2-\epsilon_{2}, \quad \forall n \in \N\cup\{0\}
\end{equation} 
for fixed, arbitrarily small, $\epsilon_{1},\epsilon_{2}>0$ satisfying $\epsilon_1+\epsilon_2\leq 2$.\newline

{\noindent\textbf{Control Sequence:}} $\{i(n)\}_{n=0}^{\infty}$ is almost cyclic in 
a generalized sense, i.e., it obeys \eqref{eq:control}.
\end{alg}

Proposition \ref{prop:0Properties}\eqref{item:NonVanishSubgrad} ensures that
$t_{n}\neq0$ whenever $g_{i(n)}(x_{n})>0$, and hence $x_{n+1}$ is
well-defined. The elements of the sequence $\{b_{n}\}_{n=0}^{\infty}$ act as
perturbation terms in the algorithm. If $b_{n}=0$ for all $n$ then the
algorithm is the ordinary feasibility-seeking Cyclic Subgradient Projection 
(CSP) algorithm of \cite{cl82}, at least when the control is almost cyclic and
the functions $g_{j}$ are convex. When the first line of (\ref{eq:iterative})
occurs, then we say that the algorithm makes an \textit{active step} at step
$n+1$. When the second line of (\ref{eq:iterative}) occurs, then we say that
the algorithm makes \textit{an inactive step} at step $n+1$.

\section{Conditions for convergence\label{sec:conditions}}

For the convergence analysis we will need the following conditions.

\begin{condition}
\label{cond:b_n} For some $\mu>0$ which is any number greater than the distance
$d(x_{0},C)$ between $x_{0}$ and $C$, the following inequality is satisfied
\begin{equation}
\parallel b_{n}\parallel\leq\min\left( \mu,\frac{\epsilon_{1}\epsilon_{2}%
h_{n}^{2}}{2(5\mu+4h_{n})}\right)  ,\quad\forall n\in\mathbb{N}, \label{eq:b^nu}
\end{equation}
where%
\begin{equation}\label{eq:h_n}
h_{n}=\left\{
\begin{array}
[c]{ll}%
g_{i(n)}(x_{n})/\Vert t_{n}\Vert, & \text{if }g_{i(n)}(x_{n})>0,\\
0, & \text{if }g_{i(n)}(x_{n})\leq0.
\end{array}
\right.
\end{equation}

\end{condition}

The construction of the sequence $\{b_{n}\}_{n=0}^{\infty}$ of perturbations 
is done in an adaptive way, in contrast to other works dealing with
inexact algorithms (such as \cite{Eckstein1998,Rockafellar1976}) in which such
terms satisfy a certain fixed (nonadaptive) condition, e.g., the summability
condition $\sum_{n=1}^{\infty}\Vert b_{n}\Vert<\infty$ or some other fixed
conditions \cite{CorvellecFlam,DES1982,SolodovSvaiter2000,SolodovSvaiter2001}. 
In our case one computes $g_{i(n)}(x_{n})$ and $h_{n}$, and then chooses any
$b_{n}$ such that \eqref{eq:b^nu} holds. The only somewhat adaptive
perturbation terms that we are aware of appear in the very recent work 
\cite[relations (31)--(32)]{OteroIusem2012}.

It is interesting to note that Condition \ref{cond:b_n} actually implies that
$\sum_{n=1}^{\infty}\Vert b_{n}\Vert<\infty$: see Remark \ref{rem:Summability} below. 
This means that Condition \ref{cond:b_n} is less general than summability,
but this is not necessarily a bad thing. Indeed, as argued briefly in \cite[p.
216]{SolodovSvaiter2000} and in a more detailed form in \cite{SvaiterPC2012}, 
the summability condition is not satisfactory 
since it gives too much freedom for the perturbations and hence it may lead to
undesired practical results. On the one hand it allows perturbations of the
form $b_{n}=n^{100n}$ for each $n\leq10^{22222}$ and $b_{n}={10000}^{-n}$ 
for each $n>10^{22222}$, which means essentially no convergence at all in
practice. On the other hand, if $b_{n}={10000}^{-n}$ right at the beginning, then
this implies that very soon the perturbations will be too small for the
computing device to make any difference as perturbations proceed (but usually this
will not accelerate the convergence). In contrast, conditions such as Condition
\ref{cond:b_n} guide the user regarding the possible values of the
perturbation at the $n$-th iteration. These values are given in terms of
previous iterations and they do not depend on future iterations as in the case of 
the summability condition. In a sense they are more adaptive to the whole
problem: they are not too large and not too small.

As a final remark concerning Condition \ref{cond:b_n}, we note that in order
to verify \eqref{eq:b^nu} one has to know $\mu$, i.e., to have an upper bound on
the (yet unknown) distance $d(x_{0},C)$. However, in practice, when applying
Algorithm  \ref{alg:perturbed-csp}, one usually restricts the problem to a
large closed, bounded, and convex region $\Omega$ (say, a cube or a ball), due to limitations in the
computing device, and the diameter of this region can be taken as $\mu$. In
other cases one may have better estimates on the value of $\mu$. For instance,
if one of the involved subsets $C_{j}$ is bounded, then $C\subseteq C_{j}$ is
bounded and one can start from a point $x_{0}$ in $C_{j}$ and take the
diameter of $C_{j}$ as $\mu$.

The second condition for convergence is the following.

\begin{condition}
\label{cond:LowerSemiCont} For each $j\in J$, the function $g_{j}$ is zero-convex,  
  uniformly continuous on closed and bounded subsets of $\Omega$, and 
weakly sequential lower semicontinuous.
\end{condition}

The condition of uniform continuity holds in many cases, e.g., when the space
is finite-dimensional (recall that $g_{j}$ is continuous with respect to the
norm topology) or when $g_{j}$ satisfies a Lipschitz or H{\"{o}}lder
condition. The weakly sequential lower semicontinuity condition holds, for
instance, when the space is finite-dimensional, or when $g_{j}$ is quasiconvex
(by \cite[Proposition 10.23]{BauschkeCombettes2011} and the assumption that
$g_{j}$ is continuous). 

The last condition that we need is the following.

\begin{condition}
\label{cond:boundedness} There exists a number $K>0$ such that $\Vert t_{n}\Vert\leq
K$ for all $n\in\mathbb{N}\cup\{0\}$.
\end{condition}

It seems that verification of Condition \ref{cond:boundedness} requires
knowledge about the functions that define the subsets of the feasibility
problem as level-sets. A possible relevant property which may help here is
that of uniform boundedness of the subgradients on bounded sets. This property
is a standard one, frequently used in theorems on subgradient projection
methods, when the functions are convex. If the space is finite-dimensional,
then it holds, e.g., if the effective domain of all functions is the whole
space and there are finitely many functions, see, e.g., \cite[Proposition 7.8
and Corollary 7.9]{bb96}. If the space is infinite-dimensional but the
functions are uniformly continuous on closed and bounded subsets (as implied
by Condition  \ref{cond:LowerSemiCont}) and all the finitely many functions are
convex, then Condition  \ref{cond:boundedness} holds too, again from
\cite[Proposition 7.8]{bb96}. When infinitely many functions are involved in
the algorithm and all of them are Lipschitz with uniformly bounded Lipschitz
constants, then Condition \ref{cond:boundedness} holds too from
\cite[Proposition 7.8]{bb96} since in this case the proof of this proposition
implies that $K$ can be any upper bound on the Lipschitz constants.

In analogy with the above we want to define the property of uniform
boundedness on bounded sets of the $0$-subgradients. However, because of
Proposition \ref{prop:0Properties}\eqref{item:fg} one cannot expect to have
uniform boundedness of all $t\in\partial^{0}g_{j}(y)$ for a given $y$ and a
given $j$, namely that $\Vert t\Vert\leq K$ for all $t\in\partial^{0}g_{j}(y)$. 
It turns out that for all our purposes it is enough that a selection of
0-subgradients will be uniformly bounded, and this is formulated in the
following definition.

\begin{definition}
\label{def:bounded} Given a family $\{g_{j}\}_{j\in J}$ of zero-convex
functions defined on $\Omega\subseteq H$, if for any bounded set
$U\subseteq\Omega$ there exists a constant $K$, called \texttt{a uniform
bound}, such that for \texttt{all} $j\in J$ and \texttt{all} $x\in U$ there
exist at least one $0$-subgradient $t\in\partial^{0}g_{j}(x)$ satisfying
$\left\Vert t\right\Vert \leq K$, then we say that the family of zero-convex
functions has the property of \texttt{partial uniform boundedness on bounded
sets of the }$0$-\texttt{subgradients}.
\end{definition}

As shown in Lemma \ref{lem:fejerM} below, the sequence $\{x_{n}\}_{n=0} 
^{\infty}$ generated by Algorithm \ref{alg:perturbed-csp} is contained in a
bounded set, independently of the assumption imposed by Condition
\ref{cond:boundedness}. Thus, if it is assumed that the family $\{g_{j}%
\}_{j\in J}$ of zero-convex functions has the partial uniform boundedness on
bounded sets of the 0-subgradients, then Condition \ref{cond:boundedness}
holds for the selection of the corresponding $0$-subgradient $t_n\in\partial
^{0}g_{i(n)}(x_n)$, $n\in\N\cup\{0\}$. Example \ref{ex:convex} (with convex functions), as well as Examples \ref{ex:polynomial}--\ref{ex:Voronoi} and  Propositions  \ref{prop:alpha}-\ref{prop:EpsilonConvex} 
show that Condition \ref{cond:boundedness} can hold in various scenarios. For instance, 
the condition holds if we assume that there is a uniform Lipschitz constant for all 
of the functions $g_j$ and then use Proposition \ref{prop:alpha} 
(with $t_n=0$ if $y\in g_{i(n)}^{\leq 0}$ and with $t_n$ defined by \eqref{eq:t} 
when $y\notin g_{i(n)}^{\leq 0}$). On the other hand, 
Remark \ref{rem:Unbounded0Subgrad} shows that this condition can be violated 
in some exotic cases. 

\section{The convergence theorem\label{sec:convergence}}

The following theorem affirms convergence of the SSP feasibility-seeking
algorithm with perturbations.

\begin{theorem}
\label{thm:resiliency} In the framework of, and under the assumptions in,
Sections \ref{sec:alg} and \ref{sec:conditions}, any sequence $\{x_{n}\}_{n=0}^{\infty},$ 
generated by Algorithm \ref{alg:perturbed-csp}, converges 
weakly to a point in the set $B[x_{0},2\mu]\cap C$, where $B[x_{0},2\mu]$ is the closed
ball of radius $2\mu$ centered at $x_{0}$ and $\mu$ is a fixed positive number greater than  $d(x_0,C)$. In addition, if either the space is
finite-dimensional or if the set $B[x_{0},2\mu]\cap C$ has a nonempty interior with respect to $H$, then the sequence converges in norm to a point in this set.
\end{theorem}

The proof of Theorem \ref{thm:resiliency} is based on the following lemmas.

\begin{lemma}
\label{lem:fejerM} In the framework of, and under the assumptions in, Sections
\ref{sec:alg} and \ref{sec:conditions}, let $q$ be any real number in the
interval $[\mu,2\mu]$ and let $x\in C$ be such that $\Vert x_{0}-x\Vert\leq q$.
Then any sequence $\{x_{n}\}_{n=0}^{\infty}$, generated by Algorithm 
\ref{alg:perturbed-csp}, is contained in $\Omega$ and has the property that 
\begin{equation}
\Vert x_{n+1}-x\Vert^{2}\leq\Vert x_{n}-x\Vert^{2}-0.5\epsilon_{1}\epsilon
_{2}h_{n}^{2}\label{eq:ActiveStep}
\end{equation}
for each $n\in\mathbb{N}\cup\{0\}$.
\end{lemma}

\begin{proof}
Simple induction shows that $x_{n}\in\Omega$ for all $n\in\mathbb{N}\cup\{0\}$. The
assumptions $C\neq\emptyset$ and $d(x_{0},C)<\mu\leq q$ imply that there does
exist an $x\in C$ such that $\Vert x_{0}-x\Vert\leq q$. Suppose that an active
step occurs at step $n+1$ and denote 
\begin{equation}\label{eq:a_n}
a_{n}:=\lambda_{n}g_{i(n)}(x_{n})/\Vert t_{n}\Vert^{2}. 
\end{equation}
Since $P_{\Omega}$ is nonexpansive 
\cite[pp. 59-61]{BauschkeCombettes2011}, the equality
$x=P_{\Omega}(x)$ and direct calculation show that
\begin{align}
\Vert x_{n+1}-x\Vert^{2} &  =\Vert P_{\Omega}(x_{n}-a_{n}t_{n}+b_{n}
)-P_{\Omega}(x)\Vert^{2}\nonumber\\
&  \leq\Vert x_{n}-a_{n}t_{n}+b_{n}-x\Vert^{2}\nonumber\\
&  =\Vert x_{n}-x\Vert^{2}+\Vert b_{n}-a_{n}t_{n}\Vert^{2}+2\langle
b_{n}-a_{n}t_{n},x_{n}-x\rangle\nonumber\\
&  =\Vert x_{n}-x\Vert^{2}+\Vert b_{n}\Vert^{2}+|a_{n}|^{2}\Vert t_{n}
\Vert^{2}-2a_{n}\langle b_{n},t_{n}\rangle\nonumber\\
&  \quad+2\langle b_{n},x_{n}-x\rangle-2a_{n}\langle t_{n},x_{n}
-x\rangle.\label{eq:4.3}%
\end{align}
Since $x\in C$ it follows that $g_{i(n)}(x)\leq0$, thus, by the 
$0$-subgradient inequality \eqref{eq:0subgradient} and the fact that $a_{n}\geq0$
we get
\begin{equation}
-2a_{n}\langle t_{n},x_{n}-x\rangle\leq-2a_{n}g_{i(n)}(x_{n}).\label{eq:4.4}%
\end{equation}
From \eqref{eq:h_n}, \eqref{eq:a_n}, \eqref{eq:4.3}, \eqref{eq:4.4},  and the Cauchy-Schwarz inequality,
\begin{align}
\Vert x_{n+1}-x\Vert^{2} &  \leq\Vert x_{n}-x\Vert^{2}+\Vert b_{n}\Vert
^{2}+a_{n}^{2}\Vert t_{n}\Vert^{2}-2a_{n}\langle b_{n},t_{n}\rangle+2\Vert
b_{n}\Vert\Vert x_{n}-x\Vert\nonumber\\
&  \quad-2a_{n}g_{i(n)}(x_{n})\nonumber\\
&  =\Vert x_{n}-x\Vert^{2}+(\lambda_{n}^{2}-2\lambda_{n})h_{n}^{2}+\Vert
b_{n}\Vert^{2}-2a_{n}\langle b_{n},t_{n}\rangle+2\Vert b_{n}\Vert\Vert
x_{n}-x\Vert.
\end{align}
By the properties of $\lambda_{n}$, the Cauchy-Schwarz inequality, the
definition of $h_{n}$, and the fact that $\Vert b_{n}\Vert\leq \mu$, we reach
\begin{align}
\Vert x_{n+1}-x\Vert^{2} &  \leq\Vert x_{n}-x\Vert^{2}-\epsilon_{1}%
\epsilon_{2}h_{n}^{2}+\Vert b_{n}\Vert(\Vert b_{n}\Vert+2\Vert x_{n}%
-x\Vert)+2a_{n}\Vert t_{n}\Vert\Vert b_{n}\Vert\nonumber\\
&  \leq\Vert x_{n}-x\Vert^{2}-\epsilon_{1}\epsilon_{2}h_{n}^{2}+\Vert
b_{n}\Vert(\mu+2\Vert x_{n}-x\Vert+4h_{n}).\label{eq:Mh_n}%
\end{align}
Now let $n=0$. If an active step occurs at step $1$, then
\eqref{eq:ActiveStep} holds because $\Vert x_{0}-x\Vert\leq q\leq2\mu$,
by (\ref{eq:Mh_n}), and by \eqref{eq:b^nu}. In particular 
$\Vert x_{1}-x\Vert\leq\Vert x_{0}-x\Vert\leq q$. If an inactive step 
occurs at step $1$, then obviously \eqref{eq:ActiveStep} holds since 
$h_{0}=0$ and $x_{n+1}=x_{n}$. In particular, $\Vert x_{1}-x\Vert\leq q$.

Continuing the induction, suppose that \eqref{eq:ActiveStep} holds up to some
$n\geq1$. If an inactive step occurs at step $n+1$ of Algorithm \ref{alg:perturbed-csp}, 
then $h_{n}=0$ and obviously \eqref{eq:ActiveStep} holds. Otherwise, since by 
the induction hypothesis $\Vert x_{n}-x\Vert\leq\cdots\leq\Vert x_{0}-x\Vert\leq q\leq2\mu$,
we obtain from \eqref{eq:Mh_n} and \eqref{eq:b^nu} the inequality 
\eqref{eq:ActiveStep}. Therefore, \eqref{eq:ActiveStep} holds for $n+1$ and
hence for every $n\in\mathbb{N}\cup\{0\}$.
\end{proof}

\begin{lemma}
\label{lem:n_0} Under the assumptions of Lemma \ref{lem:fejerM}, there exist
an integer $\nu_{0}\in\mathbb{N}$ and a real $\alpha>0$ such that%
\begin{equation}
\Vert x_{n+1}-x_{n}\Vert^{2}\leq\alpha(\Vert x_{n}-x\Vert^{2}-\Vert
x_{n+1}-x\Vert^{2}),\label{eq:estimate_C1C2}%
\end{equation}
for all $n\geq\nu_{0}$.
\end{lemma}

\begin{proof}

By \eqref{eq:ActiveStep} we have
\begin{equation}
h_{n}^{2}\leq\alpha_{1}(\Vert x_{n}-x\Vert^{2}-\Vert x_{n+1}-x\Vert
^{2}),\label{eq:h_estimate}
\end{equation}
where $\alpha_{1}=2/(\epsilon_{1}\epsilon_{2})$. The fact that $P_{\Omega
}$ is nonexpansive, the equality $x_{n}=P_{\Omega}(x_{n})$, the inequality
$|\lambda_{n}|\leq2$, \eqref{eq:h_estimate}, the Cauchy-Schwarz inequality, 
and \eqref{eq:iterative} imply that
\begin{equation}
\begin{array}
[c]{lll}%
\Vert x_{n+1}-x_{n}\Vert^{2} & \leq & \Vert b_{n}-(\lambda_{n}h_{n}t_{n}/\Vert
t_{n}\Vert)\Vert^{2}\\
& \leq & \Vert b_{n}\Vert^{2}+4\Vert b_{n}\Vert h_{n}+4\alpha_{1}(\Vert
x_{n}-x\Vert^{2}-\Vert x_{n+1}-x\Vert^{2}),
\end{array}
\label{eq:estimateC1}
\end{equation}
whenever an active step occurs. However, \eqref{eq:estimateC1} holds also when
an inactive step occurs since in that case the left-hand side is 0 and the
right-hand side is nonnegative (from Lemma {lem:fejerM}). From Lemma \ref{lem:fejerM}, the sequence 
$\{\Vert x_{n}-x\Vert\}_{n=0}^{\infty}$ is decreasing and bounded from below
and hence converges to a limit. Therefore, it is a Cauchy sequence and from
\eqref{eq:h_estimate} it follows that there exists a positive integer $\nu_{0}$ 
having the property that $h_{n}<1$ whenever $n\geq\nu_{0}$. Hence $h_{n}^{4}\leq h_{n}^{3}\leq
h_{n}^{2}$ for each $n\geq\nu_{0}$. Let $\alpha_{2}=(\epsilon_{1}\epsilon
_{2}/(10\mu))^{2}$. From \eqref{eq:b^nu} and \eqref{eq:h_estimate} it follows
that%
\begin{align}
\Vert b_{n}\Vert^{2} &  \leq(\epsilon_{1}\epsilon_{2}h_{n}^{2}/(2\cdot
(5\mu+4h_{n})))^{2}\nonumber\\
&  \leq\alpha_{2}h_{n}^{4}\leq\alpha_{2}h_{n}^{2}\leq\alpha_{1}\alpha
_{2}(\Vert x_{n}-x\Vert^{2}-\Vert x_{n+1}-x\Vert^{2}).
\end{align}
From \eqref{eq:b^nu}, the inequality $h_{n}^{3}\leq h_{n}^{2}$, and
\eqref{eq:h_estimate} it follows that%
\begin{equation}
4\Vert b_{n}\Vert h_{n}\leq2(\epsilon_{1}\epsilon_{2}/(5\mu))\alpha_{1}(\Vert
x_{n}-x\Vert^{2}-\Vert x_{n+1}-x\Vert^{2}).
\end{equation}
This and \eqref{eq:estimateC1} imply \eqref{eq:estimate_C1C2} with
\begin{equation}
\alpha=\alpha_{1}(4+\alpha_{2}+(2\epsilon_{1}\epsilon_{2}/(5\mu))), 
\end{equation}
whenever $n\geq\nu_{0}$.
\end{proof}

\begin{remark}
\label{rem:Summability} It follows from \eqref{eq:b^nu} and
\eqref{eq:h_estimate} that $\sum_{n=0}^{\infty}\Vert b_{n}\Vert<\infty$.
Indeed, from \eqref{eq:h_estimate} we have $\sum_{n=1}^{\infty}h_{n}^{2}%
\leq(2/(\epsilon_{1}\epsilon_{2}))\|x_{0}-x\|^{2}<\infty$ and from
\eqref{eq:b^nu} we have $\sum_{n=0}^{\infty}\Vert b_{n}\Vert\leq\beta
\sum_{n=1}^{\infty}h_{n}^{2}$ for some $\beta>0$.
\end{remark}

\begin{lemma}
\label{lem:n_1} Under the assumptions of Lemma \ref{lem:fejerM}, let some
$\tau>0$ be given. There exists an integer $\nu_{1}=\nu_1(\tau)\in\mathbb{N},\,\nu
_{1}\geq\nu_{0},$ where $\nu_{0}$ is from Lemma \ref{lem:n_0}, such that
$g_{i(n)}(x_{n})<\tau/2$ whenever $n\geq\nu_{1}$.
\end{lemma}

\begin{proof}
By \eqref{eq:h_estimate} and the fact that the sequence $\{\Vert x_{n}-x\Vert
\}_{n=0}^{\infty}$ is a Cauchy sequence it follows that there exists an
integer $\nu_{1}\geq\nu_{0}$ such $Kh_{n}<\tau/2$ for any $n\geq\nu_{1}$,
where $K$ is from Condition \ref{cond:boundedness}. Let $n\geq\nu_{1}$ be
given. If an inactive step occurs at step $n+1$, then 
$g_{i(n)}(x_{n})\leq0<\tau/2$. Otherwise, an active step occurs at step $n+1$. By Condition
\ref{cond:boundedness} it follows that $\Vert t_{n}\Vert\leq K$. The definition of
$h_{n}$ then implies that $g_{i(n)}(x_{n})/K\leq h_{n}$. As a result,
$g_{i(n)}(x_{n})\leq Kh_{n}<\tau/2$.
\end{proof}

\begin{lemma}
\label{lem:n_2} Under the assumptions of Lemma \ref{lem:fejerM}, let $j\in J$
and $\tau>0$ be given. Let $\nu_{1}=\nu_1(\tau)$ be taken from Lemma \ref{lem:n_1}. Then
there exists an integer $\nu_{2,j}=\nu_{2,j}(\tau)\in\mathbb{N}$, such that $\nu_{2,j}\geq
\nu_{1}$ and $|g_{j}(x_{n+s})-g_{j}(x_{n})|<\tau/2$ for all $n\geq\nu_{2,j}$
and all $s\in\{1,2,\ldots,L_{j}\}$, where $L_{j}$ is from \eqref{eq:control}.
\end{lemma}

\begin{proof}
By Lemma \ref{lem:fejerM} the sequence $\{x_{n}\}_{n=0}^{\infty}$ is contained
in the closed ball $B[x,q]$ of radius $q$ and center $x$. Since $g_{j}$ is
uniformly continuous on $B[x,q]\bigcap\Omega$ there exists a positive number $\delta_j$ such
that for all $u,v\in B[x,q]\bigcap\Omega$, if $\Vert u-v\Vert<\delta_j$ then 
$|g_{j}(u)-g_{j}(v)|<\tau/2$.

By \eqref{eq:estimate_C1C2} and the fact that the sequence $\{\Vert x_{n}-x\Vert
\}_{n=0}^{\infty}$ is (bounded below and decreasing and hence) a Cauchy sequence, it follows that there exists
$\nu_{2,j}\in\mathbb{N},\nu_{2,j}\geq\nu_{1}$ such that
\begin{equation}\label{eq:delta_jL_j}
\Vert x_{n+1}-x_{n}\Vert<\delta_j/L_{j}\quad\text{ for all }n\geq\nu_{2,j}.
\end{equation}
From \eqref{eq:delta_jL_j} and the triangle inequality it follows that 
$\Vert x_{n+s}-x_{n}\Vert<\delta_j$ for all $n\geq\nu_{2,j}$ and all integers $s\in\{1,2,\ldots,L_{j}\}$. Since $x_n,x_{n+s}\in B[x,q]\bigcap\Omega$, we conclude 
that 
\begin{equation}
|g_{j}(x_{n+s})-g_{j}(x_{n})|<\tau/2\text{ whenever }n\geq\nu_{2,j}.
\end{equation}

\end{proof}

\begin{lemma}
\label{lem:cluster} Under the assumptions of Lemma \ref{lem:fejerM}, any weak
cluster point of a sequence $\{x_{n}\}_{n=0}^{\infty},$ generated by Algorithm 
\ref{alg:perturbed-csp}, belongs to $C$.
\end{lemma}

\begin{proof}
Suppose that $y\in H$ is a weak cluster point of $\{x_{n}\}_{n=0}^{\infty}$,
i.e., a subsequence $\{x_{n_{k}}\}_{k=0}^{\infty}$ of $\{x_{n}\}_{n=0}^{\infty}$ 
converges weakly to $y$.

Let $j\in J$ and $\tau>0$ be given. Let $k$ be large enough so that 
$n_{k}>\nu_{2,j}$ where $\nu_{2,j}$ is from Lemma \ref{lem:n_2}. Since the control
sequence satisfies \eqref{eq:control} there exists an integer $s\in\lbrack
n_{k},L_{j}-1+n_{k}]$ such that $i(s)=j$. From Lemma \ref{lem:n_2} we know
that 
\begin{equation}
g_{j}(x_{n_{k}})-g_{j}(x_{s})<\tau/2.
\end{equation}
Consequently, if $g_{j}(x_{s})\leq0$, then 
\begin{equation}
g_{j}(x_{n_{k}})=g_{j}(x_{n_{k}})-g_{j}(x_{s})+g_{j}(x_{s})<\tau/2+0.
\label{eq:g_jr2}%
\end{equation}
If $g_{j}(x_{s})>0$, then an active step occurs at step $s+1$. Since
$s\geq n_{k}>\nu_{2,j}\geq\nu_{1}$ and $j=i(s)$, it follows from the
definitions of $\nu_{1},\nu_{2,j}$ and from Lemma \ref{lem:n_1} that 
$g_{j}(x_{s})<\tau/2$. Hence, as in \eqref{eq:g_jr2}, we have
\begin{equation}
g_{j}(x_{n_{k}})<\tau/2+\tau/2=\tau.
\end{equation}
Therefore, from the weakly sequential lower semicontinuity of $g_{j}$ we
conclude that the inequality $g_{j}(y)\leq\liminf_{k\rightarrow\infty}g_{j}(x_{n_{k}} 
)\leq\tau$ holds for each $\tau>0$. As a result, $g_{j}(y)\leq0$ for each $j\in J$
and, thus, $y\in C$.
\end{proof}

In order to prove Theorem \ref{thm:resiliency} we need one of the following
two general lemmas.  
\begin{lemma}
\label{lem:WeakLimit} Suppose that $\{x_{n}\}_{n=0}^{\infty}$ is a bounded
sequence and that the limit  $\lim_{n\rightarrow\infty}\Vert x_{n}-z\Vert$ 
exists for each weak limit point $z$ of the sequence. Then the 
whole sequence converges weakly.
\end{lemma}

Lemma \ref{lem:WeakLimit}  is a particular case of \cite[Lemma 3.4]%
{ReemBregmanII2012} (take there $X$ to be a Hilbert space, $\mathbb{T}$ to be
the weak topology, $D(x,y)=\Vert x-y\Vert$ or $D(x,y)=\Vert x-y\Vert^{2}$, and
also use \cite[Example 2.5]{ReemBregmanII2012} or \cite[Example 2.6]%
{ReemBregmanII2012}). Lemma \ref{lem:WeakLimit} can also be deduced, after
some manipulations, from \cite[Theorem 4.2]{GoebelReich} or from the proof of
\cite[Proposition 1(iii)]{AlberIusemSolodov1998}.

\begin{lemma}
\label{lem:BrowderOpialBB} Let $F$ be a closed and convex subset of a Hilbert
space $H$, and suppose that $\{x_{n}\}_{n=0}^{\infty}$ is a bounded sequence 
in $H$ such that
 
\begin{enumerate}
[(a)]

\item \label{item:Fejer} $\{x_{n}\}_{n=0}^{\infty}$ is Fej\'er monotone with respect to $F$, that is, the sequence $\{\Vert x_{n}-x\Vert\}_{n=0}^{\infty}$ 
is decreasing for each $x\in F$.

\item Each weak cluster point of the sequence $\{x_{n}\}_{n=0}^{\infty}$ lies 
in $F$.
\end{enumerate}

Then $\{x_{n}\}_{n=0}^{\infty}$ converges weakly to a point in $F$. Alternatively, 
if \eqref{item:Fejer} holds and the interior of $F$ is nonempty, 
then the sequence converges strongly to a point in $H$.
\end{lemma}

The weak convergence part of 
Lemma \ref{lem:BrowderOpialBB} is from either \cite[Lemma 6]{Browder1967} (but, 
as noted in \cite{Browder1967}, this lemma was essentially proved by Opial 
in \cite[Lemma 1]{Opial1967}) or \cite[Theorem 2.16(ii)]{bb96}. The strong convergence part 
is from \cite[Theorem 2.16(iii)]{bb96}. 

It is interesting to note that both lemmas hold in a more general context:
Lemma \ref{lem:WeakLimit} holds in the general setting of weak-strong spaces
and corresponding Bregman distances without Bregman functions, while Lemma 
\ref{lem:BrowderOpialBB} can be generalized to uniformly convex Banach spaces
having a weakly continuous duality mapping \cite[Lemma 11]{Browder1967} (see
also \cite[Lemma 3]{Opial1967}). In addition, as observed in \cite[Theorem
5.5, Proposition 5.10]{BauschkeCombettes2011}, the subset $F$ does not have to
be closed and convex but rather it can be arbitrary nonempty (or,
respectively, with a nonempty interior) when the space is Hilbert. In fact, as
observed \cite[Proposition 3.10]{Combettes2001}, in this case $\{x_{n}\}_{n=0}^{\infty}$
 may be just quasi-Fej\'er.

Now we are ready to prove Theorem \ref{thm:resiliency}.

\begin{proof}
[\textbf{proof of Theorem \ref{thm:resiliency}}] Let $x\in C$ be such that
$d(x,x_{0})<\mu$. There exists such an $x$ since $d(x_{0},C)<\mu$. From Lemma
 \ref{lem:fejerM} (with this $x$) it follows that $\{x_{n}\}_{n=0}^{\infty}$ is 
contained in the ball $B[x,\mu]$. Hence it has at least one weak cluster point.
Any weak cluster point $y$ of the sequence belongs to this ball since by the
lower semicontinuity of the norm we have $\Vert y-x\Vert\leq\liminf_{n\rightarrow
\infty}\Vert x_{n}-x\Vert\leq \mu$. From Lemma \ref{lem:cluster} we know that
$y\in C$. In addition, since
\begin{equation}
\Vert x_{0}-y\Vert\leq\Vert x_{0}-x\Vert+\Vert x-y\Vert\leq2\mu
\end{equation}
we can apply Lemma  \ref{lem:fejerM} with $y$ instead of $x$ to conclude that
the sequence $\{\Vert x_{n}-y\Vert\}_{n=0}^{\infty}$ of 
nonnegative numbers is decreasing and hence converges to a nonnegative number. The previous consideration holds
for any weak limit point. As a result, Lemma \ref{lem:WeakLimit} ensures that
$\{x_{n}\}_{n=0}^{\infty}$ converges weakly to some point, which, by Lemma
 \ref{lem:cluster}, belongs to $C$, and, by the above, to $F:=B[x_{0},2\mu]\cap
C$.

Alternatively, the above already proves that any weak cluster point $y$ of the
sequence belongs to the closed and convex subset $F$ and also that the
nonnegative sequence $\{\Vert x_{n}-y\Vert\}_{n=0}^{\infty}$ is decreasing.
Hence, from Lemma \ref{lem:BrowderOpialBB} it follows that $\{x_{n}\}_{n=0}^{\infty}$ 
converges weakly to some point in $F$. By the same lemma,
the convergence is strong if $F$ has a nonempty interior. The corresponding
limit point is in $F$ since it coincides with the unique weak limit point
which is there. The strong convergence holds also if the space is 
finite-dimensional since in this case the weak and strong topologies coincide.
\end{proof}

\section{Computational results}\label{sec:ComputationalResults}
In this section we present a concrete example of the CFP with 
zero-convex functions, together with relevant computational results. 
The example is related to Examples \ref{ex:Voronoi} and \ref{ex:VoronoiWeighted} above. 
The context is molecular biology. 

\subsection{The setting} 
The setting of the example is as follows. There is 
 a material located in a 3D box $\Omega$ and composed of two 
 types of molecules. Each molecule type is modeled by a ball. One 
 type has radius $r$ and the other has radius $R>r$, measured in 
 angstroms (\AA). This scenario  
 is common in molecular biology \cite{GTL1995,GPF1997} where the 
 first molecule typw is water ($r=1.4$\AA)  and the second type  is 
a material which comes in contact with water such as some 
compounds of a protein (on the protein surface). An example 
of such a material is alpha carbon (CA) whose radius is $R=1.87$\AA. 
As explained in Example \ref{ex:VoronoiWeighted} above and the references 
therein, the additively weighted Voronoi cell  of a given molecule  plays 
an  important role in this context. 

Consider now a water molecule whose center is $p$. Denote by 
$V_p$ its  additively weighted Voronoi cell. 
We look for a point in $V_p$ which is not too far from $p$ 
and not too far from a certain neighboring alpha carbon molecule. 
In other words, we want to find a point in the intersection of $V_p$ and 
two balls. Such a point may help in trimming parts of the interaction 
interface using a spherical probe  \cite{KWCKLBK2006}. 

In what follows we formulate the problem as a 
convex feasibility problem.  
Let the locations of all molecules different from $p$ be denoted by the 
3-dimensional vectors $a_0,a_1,\ldots,a_{\ell}$. 
Let $I_{\textnormal{water}}=
\{0,1,\ldots,j_{\textnormal{water}}\}$ be the set of indices of water molecules and 
$I_{\alpha}=\{j_{\textnormal{water}}+1,j_{\textnormal{water}}+2,\ldots,\ell\}$ be the set of indices of the alpha carbon molecules. 
For each $j\in I:=\{0,1,\ldots,\ell\}=I_{water}\bigcup I_{\alpha}$ let $w_j=r$ when molecule $j$ is a water molecule 
and $w_j=R$ when this molecule is alpha carbon. 
From Example \ref{ex:VoronoiWeighted} we know that 
\begin{multline}
V_p=\{x\in \Omega\, |\, d(x,p)-r\leq d(x,a_j)-w_j,\,j\in I\}
=\bigcap_{j\in I}\{x\in \Omega\, |\, d(x,p)-r\leq d(x,a_j)-w_j\}\\
=\left(\bigcap_{j\in I_{\textnormal{water}}}\{x\in \Omega\, |\, d(x,p)\leq d(x,a_j)\}\right)
\bigcap\left(\bigcap_{j\in I_{\alpha}}\{x\in \Omega\, |\, d(x,p)-d(x,a_j)+R-r\leq 0\}\right).
\end{multline}
Given $j\in I_{\textnormal{water}}$ let $C_j=\{x\in \Omega\,|\, d(x,p)\leq d(x,a_j)\}$. 
This set is the intersection of a half-space and $\Omega$ and it can be written as  
$C_j=g_j^{\leq 0}$ 
where $g_j:\Omega\to\R$ is the function defined by $g_j(x)=\langle x-0.5(a_j+p),(a_j-p)/\|a_j-p\|\rangle$. 
Given $j\in I_{\alpha}$, let $g_j:\Omega\to\R$ be defined 
by $g_j(x)=d(x,p)-d(x,B_j)$, where $B_j$ is the closed ball of radius $R-r$ 
around $a_j$, and let $C_j=\{x\in \Omega\,|\, d(x,p)-d(x,a_j)+R-r\leq 0\}$.  
Example \ref{ex:VoronoiWeighted} above shows that $C_j=g_j^{\leq 0}$. Finally, 
define two additional functions $g_{\ell+1},g_{\ell+2}:\Omega\to\R$ by $g_{\ell+1}(x)=d(x,p)-\rho$ and 
$g_{\ell+2}(x)=d(x,a_{\ell})-\rho$, where $\rho$ is the radius of the probe 
and $a_{\ell}$ is the location of the alpha carbon molecule mentioned earlier and related to the probe. 
Let $C_j=g_j^{\leq 0}$, $j=\ell+1,\ell+2$, and let  $J=\{0,1,2,\ldots,\ell+2\}$. 
Our goal is to find a point in the set 
\begin{equation}
C:=V_p\bigcap C_{\ell+1} \bigcap C_{\ell+2}=\bigcap_{j \in J}C_j.
\end{equation}
Example \ref{ex:Voronoi} above ensures 
that $g_j$ is zero-convex (and continuous) for each $j\in J$. Hence  
 $C_j$ is closed and convex 
for all $j\in J$. For the selection of the 0-subgradients   
it suffices to consider $y\notin g_j^{\leq 0}$ and to divide the discussion 
into several cases.  If $j=\ell+1$, 
then $g_j$ (and its extension to $\R^3$ defined by the same formula) is convex 
and since it is smooth at $y$ we can take $t=\nabla g_j(y)=(y-p)/\|y-p\|$. 
The norm of $t$ is bounded by 1. 
In the same way we can take $t=(y-a_{\ell})/\|y-a_{\ell}\|$ when $j=\ell+2$. 
If $j\in I$, then  we can use \eqref{eq:t_voronoi} with  
$a=a_j+(w_j-r)(y-a_j)/\|y-a_j\|$ if $y\notin B_j$ and $a=y$ otherwise, because 
this $a$ satisfies $d(y,B_j)=d(y,a)<d(y,p)$ (we denote 
$B_j:=\{a_j\}$ when $j\in I_{\textnormal{water}}$). 
According to Example \ref{ex:Voronoi}, the norms of the resulting 0-subgrdients 
are bounded by 2. 

%
%Theorem \ref{thm:resiliency} ensures that if we construct an iterative sequence 
%according to Algorithm \ref{alg:perturbed-csp} with relaxation parameters satisfying 
%\ref{eq:relaxation}, control satisfying \eqref{eq:control},  
%and perturbations satisfying Condition \ref{cond:b_n}, then the sequence will 
%converge to a point in $C$. 
%
Regarding  the locations 
of the molecules, we assume that they roughly form a two-sided arrangement, where the 
CA molecules are in one side of the cube $\Omega$, and the water molecules are in 
another side of $\Omega$. The molecule located at $p$ is in the middle of the cube, 
namely, $p=(0,0,0)$.
It may happen that this configuration 
of locations is not likely 
to be realized (or will not be stable), since these data are not taken 
from measurements or from related computer experiments. However, different locations of the molecules  
will merely result in different values of some parameters but will usually not 
affect the essential properties of the setting (zero-convexity of the functions, etc.). 
The main goal of this 
example is to illustrate the methods 
and concepts discussed in this paper. To see that the algorithm really works 
also in other configurations, we made simulations in the case of 
random configurations of molecules in 3D and higher dimensions. See Table \ref{table:HigherDimensions} below. 

\subsection{Concrete values in the simulations} 
In the concrete simulations the box was $\Omega=[-4,4]^3$ (in the higher dimensional version 
of the problem we took $\Omega=[-4,4]^{\dim}$). There were 16 water molecules 
located at 
$a_0=(3.5, -3.5, -3.5)$, $a_1=(3.5, 0.0, -3.5)$, $a_2=(3.5, 3.5, -3.5)$, 
$a_3=(3.5, -3.5, 0.0)$, $a_4=(3.5, 0.0, 0.0)$, $a_5=(3.5, 3.5, 0.0)$, 
$a_6=(3.5, -3.5, 3.5)$, $a_7=(3.5, 0.0, 3.5)$, $a_8=(3.5, 3.5, 3.5)$, 
$a_9=(0.0, -3.5, -3.5)$, $a_{10}=(0.0, 0.0, -3.5)$, $a_{11}=(0.0, 3.5, -3.5)$,
$a_{12}=(0.0, -3.5, 0.0)$, $a_{13}=(0.0, 3.5, 0.0)$, $a_{14}=(0.0, -3.5, 3.5)$,
$a_{15}=(0.0, 3.5, 3.5)$, and 10 CA molecules located at  
$a_{16}=(-3.5, -3.5, -3.5)$,
$a_{17}=(-3.5, 0.0, -3.5)$,
$a_{18}=(-3.5, 3.5, -3.5)$,
$a_{19}=(-3.5, -3.5, 0.0)$,
$a_{20}=(-3.5, 0.0, 0.0)$,
$a_{21}=(-3.5, 3.5, 0.0)$,
$a_{22}=(-3.5, -3.5, 3.5)$,
$a_{23}=(-3.5, 0.0, 3.5)$,
$a_{24}=(-3.5, 3.5, 3.5)$,
$a_{25}=(0.0, 0.0, 3.5)$.  The maximum index was therefore $\ell=25$ and the 
total number of functions was $28=\ell+3=:\ell_3$.  

For the stopping condition, 
we defined a variable called ``smallNumber'' and checked every period  
that $g_j(x_n)\leq smallNumber$ 
for all $j\in J$, namely that $x_n$ is in the $smallNumber$-level set of $g_j$ 
for all $j\in J$. We took $smallNumber=0.00001$. 
When the control was cyclic, the period mentioned above was the length of a cycle, namely 
$\ell_3$ (the total number of functions). 
When the control was almost cyclic, the period was $3\ell_3$ as explained below. 
If the number of iterations exceeded a certain large number chosen by the user 
($5\cdot 10^6$ in our case) without finding a feasible point, then the process 
stopped with an output saying this.  
 
The almost cyclic control was constructed in the following way. First, we constructed  
an array called \emph{almost$\_$cycle} of length $2\ell_3$ (starting from 0) whose first 
$\ell_3$ entires were selected randomly 
from $\{0,1\}$. For $k\in \{\ell_3,\ell_3+1,\ldots,2\ell_3-1\}$, entry number $k$ was $1-$\emph{almost$\_$cycle}$[k-\ell_3]$. 
We constructed the control $i(n)$ as follows: when both   
\emph{almost$\_$cycle}$[n\mod (2\ell_3)]=1$ and $(n\mod 2\ell_3) \in \{0,1,\ldots,\ell_3-1\}$ held 
true, then $i(n)$ was $n\mod (2\ell_3)$. When both \emph{almost$\_$cycle}$[n\mod (2\ell_3)]=1$ and  
 $(n\mod 2\ell_3)\in \{\ell_3,\ell_3+1,\ldots,2\ell_3-1\}$ held true, we had 
 $i(n)=(n\mod 2\ell_3)-\ell_3$. Otherwise (namely, when \emph{almost$\_$cycle}$[n\mod (2\ell_3)]=0)$ the control value $i(n)$ was selected randomly from $\{0,1,\ldots,\ell_3-1\}$. A simple checking 
(which merely needs to take into account the case \emph{almost$\_$cycle}$[n\mod (2\ell_3)]=1$)  
shows that every index $j\in J=\{0,1,\ldots,\ell+2\}$ is selected at least once  
in any block  of nonnegative consecutive integers whose length is at least $3\ell_3$. 
Thus, this control is indeed almost cyclic with period $3\ell_3$.

For the perturbations, we constructed random vectors whose length 
is the right-hand side of \eqref{eq:b^nu}. The user could also choose to perform a 
calculation with zero perturbations. 

For the relaxation parameters, we either took $\lambda_n=\epsilon_1$ for all $n$, or 
 $\lambda_n=2-\epsilon_2$ for all $n$, or $\lambda_n=0.5(\epsilon_1+2-\epsilon_2)$ for all $n$, or 
 $\lambda_n$=a random number in the interval $[\epsilon_1,2-\epsilon_2]$ for all $n$.
 
\subsection{The computational results} The tables below describe the computational 
results. Here is a legend of abbreviation that are used: no.=the serial number of each experimental run of the algorithm; perturb=the perturbation terms were nonzero; 
ac=almost cyclic; c=cyclic; min numb. iter.=minimum number of iterations among 10 trials; max numb. iter.=maximum number of iterations among 10 trials; aver. numb. iter.=average number of iterations among 10 trials; feasible point: the feasible point obtained after the specified number of iterations in the minimum case; dim=dimension.

% Table generated by Excel2LaTeX 
\begin{table}[htbp]
  \centering
  \caption{Two-sided 3D arrangement, $x_0=(4,3.853,4)$, $\rho=2.0318$}
 
     \begin{tabular}{|p{0.4cm}|l|l|l|p{0.9cm}|p{0.96cm}|p{0.9cm}|p{0.9cm}|p{0.9cm}|p{2.6cm}|}
    \toprule
    \textbf{no.}& \boldmath{}\textbf{$\epsilon_1$}\unboldmath{} & \boldmath{}\textbf{$\epsilon_2$}\unboldmath{} & \textbf{$\lambda_n$} & \textbf{control} & \textbf{perturb} & \textbf{min numb. iter.} & 
    \textbf{max numb. iter.} & \textbf{aver. numb. iter.} &\textbf{feasible point} \\
 
    \midrule
   1      & 0.303 & 0.57  & $1.43$ & ac    & no    & 84    & 2688  & 621.6 & $(-0.053,0.375,1.504)$ \\
    
      2     & 0.303 & 0.57  & $0.303$ & ac    & no    & 25788 & 26880 & 26342.4 & $(0.288,0.283,1.509)$ \\
    3      & 0.303 & 0.57  & random & c     & no    & 5404 & 5880 & 5656 & $(-0.030,0.403,1.509)$ \\
    4     & 0.303 & 0.57  & $0.88$ & c     & no    & 6104  & 6104  & 6104  & $(-0.003,0.404,1.509)$ \\
    5      & 0.303 & 0.57  & $1.43$ & c     & no    & 1764  & 1764  & 1764  &  $(-0.310,0.258,1.509)$ \\
    6      & 0.303 & 0.57  & $0.303$ & c     & no    & 25368 & 25368 & 25368 & $(0.263,0.306,1.509)$ \\
    7      & 0.303 & 0.57  & random & ac    & no    & 168   & 8064  & 6745.2 & $(0.340,0.193,1.508)$ \\
    8      & 0.303 & 0.57  & random & ac    & yes   & 7476  & 8316  & 7845.6 & $(0.338,0.221,1.509)$ \\
    9      & 0.303 & 0.57  & $1.43$ & ac    & yes   & 168   & 2688  & 1142.4 & $(0.029,0.143,1.503)$ \\
    10     & 0.303 & 0.57  & $0.88$ & c     & yes   & 6104  & 6104  & 6104  & $(-0.003,0.404,1.509)$ \\
    11     & 0.303 & 0.57  & $1.43$ & c     & yes   & 1764  & 1764  & 1764  & $(-0.31,0.258,1.509)$ \\
    12     & 0.303 & 0.57  & $0.303$ & c     & yes   & 25368 & 25368 & 25368 & $(0.264,0.306,1.509)$ \\
    13     & 1     & 1     & $1$ & c     & no    & 4676  & 4676  & 4676  & $(-0.090, 0.397, 1.509)$ \\
    14     & 1     & 1     & $1$ & c     & yes   & 4676  & 4704  & 4678.8 & $(-0.089,0.394,1.509)$ \\
    15     & 1     & 1     & $1$ & ac    & no    & 6804  & 7644  & 7341.6 & $(0.199,0.351,1.509)$ \\
    16     & 1     & 1     & random & ac    & yes   & 6804  & 7644  & 7257.6 & $(0.198,0.352,1.509)$ \\
    17     & 0.1   & 1.9   & $0.1$ & c     & no    & 84924 & 84924 & 84924 & $(0.285,0.286,1.509)$ \\
    18     & 0.1   & 1.9   & $0.1$ & c     & yes   & 84924 & 84924 & 84924 & $(0.285,0.286,1.509)$ \\
    19     & 0.01  & 1.99  & $0.01$ & c     & no    & 884772 & 884772 & 884772 & $(0.289,0.282,1.509)$ \\
    20     & 0.01  & 1.99  & $0.01$ & c     & yes   & 884772 & 884772 & 884772 & $(0.289,0.281,1.509)$ \\
    21     & 1.9   & 0.1   & $1.9$ & c    & no    & 168   & 168   & 168   & $(-0.051, 0.057, 1.498)$ \\
    22     & 1.9   & 0.1   & $1.9$ & c     & yes   & 168   & 168   & 168   & $(-0.051,0.057,1.498)$ \\
    23     & 1.99  & 0.01  & $1.99$ & c     & no    & 308   & 308   & 308   & $(-0.001, 0.001, 1.470)$ \\
    24     & 1.99  & 0.01  & $1.99$ & c     & yes   & 308   & 308   & 308   & $(0.000,0.000,1.470)$ \\
    25     & 1.95  & 0.01  & $1.95$ & c     & no    & 224   & 224   & 224   & $(-0.011, 0.013, 1.469)$ \\
    26     & 1.95  & 0.01  & $1.95$ & c     & yes   & 224   & 224   & 224   & $(-0.011,0.013,1.469)$ \\
    27     & 1.95  & 0.01  & $1.99$ & c     & no    & 252   & 252   & 252   & $(-0.004, 0.004, 1.484)$ \\
    28     & 1.95  & 0.01  & $1.99$ & c     & yes   & 308   & 308   & 308   & $(0.000,0.000,1.470)$ \\
    29     & 1.95  & 0.01  & $1.97$ & c     & no    & 252   & 252   & 252   & $(-0.004, 0.004, 1.485)$ \\
    30     & 1.95  & 0.01  & $1.97$ & c     & yes   & 252   & 252   & 252   & $(-0.004,0.004,1.484)$ \\
    31     & 1.95  & 0.01  & random & c     & no    & 252   & 280   & 254.8 & $(-0.623,0.747,0.730)$ \\
    32     & 1.99  & 0.01  & $1.99$ & ac    & no    & 168   & 504   & 302.4 & $(0.008,0.000,1.476)$ \\
    33     & 0.01  & 1.99  & $0.01$ & ac    & no    & 863016 & 883764 & 879018 & $(0.293,0.278,1.509)$ \\
    34     & 1.4   & 0.6   & $1.4$ & c     & no    & 1932  & 1932  & 1932  & $(-0.304, 0.265, 1.509)$ \\
    35     & 1.4   & 0.6   & $1.4$ & c     & yes   & 1932  & 1932  & 1932  & $(-0.304,0.265,1.509)$ \\
    36     & 0.6   & 1.4   & $0.6$ & c     & no    & 10752 & 10752 & 10752 & $(0.151,0.374,1.509)$ \\
    37     & 0.6   & 1.4   & $0.6$ & c     & yes   & 10752 & 10752 & 10752 & $(0.156,0.373,1.509)$ \\
    38     & 0.7   & 1.3   & $0.7$ & c     & no    & 8596  & 8596  & 8596  & $(0.097,0.392,1.509)$ \\
    39     & 1.95  & 0.05  & $1.95$ & c     & no    & 224   & 224   & 224   & $(-0.011, 0.013, 1.470)$ \\
    40     & 1.96  & 0.04  & $1.96$ & c     & no    & 252   & 252   & 252   & $(-0.007, 0.008, 1.513)$ \\
    41     & 2.02  & 0.1   & $2.02$ & c     & no    & 448   & 448   & 448   & $(0,0,1.473)$ \\
    42     & 2.02  & 1.4   & $2.02$ & c     & no    & 448   & 448   & 448   & $(0,0,1.473)$ \\
    43     & 2.02  & 1.4   & $2.02$ & ac    & no    & 84    & 504   & 289.3 & $(-0.123,0.140,1.494)$ \\
    44     & 1.4   & 1.4   & $1.4$ & c     & yes   & 1932  & 1932  & 1932  & $(-0.304,0.265,1.509)$ \\
    45     & 1.7   & 0.2   & $1.7$ & c     & yes   & 140   & 168   & 148.4 & $(-0.271,0.201,1.510)$ \\

    \bottomrule
    \end{tabular}%
  \label{table:3Dx0PositivemmRho}
\end{table}%

% Table generated by Excel2LaTeX 
\begin{table}[htbp]
  \centering
  \caption{Two-sided 3D arrangement, $x_0=(-4,3.853,-4)$}
   \begin{tabular}{|p{0.4cm}|l|l|l|p{0.5cm}|p{0.9cm}|p{0.96cm}|p{0.9cm}|p{0.9cm}|p{0.9cm}|p{2.7cm}|}
  
    \toprule
    \textbf{no.}&{\boldmath{}\textbf{$\rho$}\unboldmath{}} & \boldmath{}\textbf{$\epsilon_1$}\unboldmath{} & \boldmath{}\textbf{$\epsilon_2$}\unboldmath{} & \textbf{$\lambda_n$} & \textbf{control} & \textbf{perturb} & \textbf{min numb. iter.} & 
    \textbf{max numb. iter.} & \textbf{aver. numb. iter.} &\textbf{feasible point} \\
  \midrule
  
   1     & 3     & 0.02  & 1.5   & $0.02$ & ac    & yes   & 17136 & 18228 & 17816.4 & $(-0.908,0.984,0.815)$ \\
   
    2     & 3     & 0.02  & 1.5   & $0.02$ & c     & no    & 17724 & 17724 & 17724 & $(-0.921, 0.986, 0.821)$ \\
    3     & 3     & 0.02  & 1.5   & $0.02$ & c     & yes   & 17724 & 17724 & 17724 & $(-0.925,0.983,0.821)$ \\
    4     & 3     & 0.7   & 1.5   & $0.7$ & c     & no    & 280   & 280   & 280   & $(-1.163, 0.998, 0.921)$ \\
    5     & 3     & 0.7   & 1.5   & $0.7$ & c     & yes   & 280   & 280   & 280   & $(-1.166,0.988,0.919)$ \\
    6     & 3     & 1.7   & 0.2   & $1.7$ & c     & no    & 28    & 28    & 28    & $(-0.448, 0.359, 0.567)$ \\
    7     & 3     & 1     & 1     & $1$ & c     & no    & 28    & 28    & 28    & $(-1.137, 1.098, 0.950)$ \\
    8     & 3     & 1     & 1     & $1$ & c     & yes   & 28    & 56    & 42    & $(-1.153,1.088,0.954)$ \\
    9     & 3     & 1.99  & 0.01  & $1.99$ & c     & yes   & 28    & 28    & 28    & $(-0.380,0.228,0.713)$ \\
    10    & 2.0318 & 1.7   & 0.2   & $1.7$ & c     & yes   & 140   & 140   & 140   & $(-0.103,0.080,1.472)$ \\
    11    & 2.0318 & 1.7   & 0.2   & $1.7$ & c     & no    & 112   & 112   & 112   & $(-0.104, 0.083, 1.473)$ \\
    12    & 2.0318 & 1.4   & 0.6   & $1.4$ & c     & no    & 1736  & 1736  & 1736  & $(-0.283, 0.288, 1.509)$ \\
    13    & 2.0318 & 1.4   & 0.6   & $1.4$ & c     & yes   & 1708  & 1764  & 1744.4 & $(-0.278,0.292,1.509)$ \\
    14    & 2.0318 & 1     & 1     & $1$ & c    & yes   & 4704  & 4704  & 4704  & $(-0.286,0.285,1.509)$ \\
    15    & 2.0318 & 1     & 1     & $1$ & c     & no    & 4704  & 4704  & 4704  & $(-0.290, 0.281, 1.509)$ \\
    16    & 2.0318 & 0.1   & 1.9   & $0.1$ & c     & no    & 84224 & 84224 & 84224 & $(-0.282, 0.289, 1.509)$ \\
    17    & 2.0318 & 0.1   & 1.9   & $0.1$ & c     & yes   & 84168 & 84280 & 84218.4 & $(-0.281,0.290,1.509)$ \\
    18    & 2.0318 & 1.9   & 1.9   & $1.9$ & c     & yes   & 168   & 196   & 170.8 & $(-0.014,0.006,1.504)$ \\
    19    & 2.0318 & 1.9   & 1.9   & $1.9$ & c     & no    & 168   & 168   & 168   & $(-0.022, 0.011, 1.477)$ \\
    20    & 2.0318 & 1.9   & 1     & $1.9$ & c     & no    & 168   & 168   & 168   & $(-0.022, 0.011, 1.477)$ \\
    21    & 1.5   & 1.9   & 0.1   & $1.9$ & c     & no    & $5\cdot 10^6$ & $5\cdot 10^6$ & $5\cdot 10^6$ & not found \\
    22    & 1     & 1.9   & 0.1   & $1.9$ & c     & no    & $5\cdot 10^6$ & $5\cdot 10^6$ & $5\cdot 10^6$ & not found \\

     \bottomrule
   
    \end{tabular}%
  \label{table:3Dx0Negative}%
\end{table}%

% Table generated by Excel2LaTeX 
\begin{table}[htbp]
  \centering
  \caption{random configurations in various dimensions}
    \begin{tabular}{|l|l|l|p{0.4cm}|p{0.4cm}|p{1cm}|p{1cm}|p{0.9cm}|p{1.4cm}|p{1cm}|p{0.7cm}|}
    \toprule
    \textbf{no.}&{\boldmath{}\textbf{$\rho$}\unboldmath{}} & \boldmath{}\textbf{$\epsilon_1$}\unboldmath{} & \boldmath{}\textbf{$\epsilon_2$}\unboldmath{} & \textbf{$\lambda_n$} & \textbf{control} & \textbf{perturb} & \textbf{min numb. iter.} & 
    \textbf{max numb. iter.} & \textbf{aver. numb. iter.} &\textbf{dim} \\
    
    \midrule
     1     & 75    & 1.99  & 0.01  & $1.99$ & c     & yes   & 56    & 84    & 67.2  & 2500 \\
  
    2     & 180   & 1.99  & 0.01  & $1.99$ & c     & yes   & 0     & 0     & 0     & 2500 \\
    3     & 60    & 1.99  & 0.01  & $1.99$ & c     & yes   & 224   & 476   & 364 & 2500 \\
    4     & 59  & 1.99  & 0.01  & $1.99$ & c     & yes   & 336 & 1428  & 638.4  & 2500 \\
    5     & 59  & 1.99  & 0.01  & $1.99$ & c     & no    & 392  & 980  & 616  & 2500 \\
    6     & 13    & 1.99  & 0.01  & $1.99$ & c     & yes   & 84    & 280   & 128.8 & 100 \\
    7     & 13    & 1.5   & 0.4   & $1.5$ & c     & no    & 56    & 84    & 64.4  & 100 \\
    8     & 13    & 1.7   & 0.3   & $1.7$ & c     & no    & 84    & 140   & 95.2  & 100 \\
    9     & 40    & 1.6   & 0.4   & $1.6$ & c     & no    & 56    & 84    & 81.2  & 1000 \\
    10    & 50    & 1.9   & 0.1   & $1.9$ & c     & no    & 56    & 56    & 56    & 1000 \\
    11    & 50    & 1.9   & 0.1   & $1.9$ & c     & no    & 56    & 56    & 56    & 1000 \\
    12    & 40    & 1     & 1     & $1$ & c     & yes   & 1680  & 2660  & 2063.6 & 1000 \\
    13    & 3     & 1     & 1     & $1$ & c     & no    & 28    &  $5\cdot10^6$  & 500151.2 & 3 \\
    14    & 3     & 1     & 1     & $1$ & c     & yes   & 28    & 184996 & 18743.2 & 3 \\
    15    & 3     & 1     & 1     & $1$ & ac    & no    & 84    & 1344  & 294   & 3 \\
    16    & 3     & 1.99  & 0.01  & $1.99$ & c     & no    & 28    & 112   & 53.2  & 3 \\
    17    & 3     & 1.99  & 0.01  & $1.99$ & ac    & no    & 84    & 84    & 84    & 3 \\
    18    & 3     & 0.01  & 1.99  & $0.01$ & c     & no    & 28504 & 82852 & 46015.2 & 3 \\
    19    & 3     & 0.01  & 1.99  & $0.01$ & ac    & no    & 26544 & 112392 & 47292 & 3 \\
    20    & 2.0318 & 0.01  & 1.99  & $0.01$ & c     & no    & 863240 &  $5\cdot10^6$ & 2526148 & 3 \\
    21    & 2.0318 & 1.99  & 0.01  & $1.99$ & c     & no    & 56    & $5\cdot10^6$ & 1500210 & 3 \\

   \bottomrule
   
    \end{tabular}%
  \label{table:HigherDimensions}%
\end{table}%

\subsection{Discussion}
The results show  that usually the perturbation terms have little 
influence on both the number of iterations and the obtained feasible point  
(see, e.g, line 13 and beyond in Table \ref{table:3Dx0PositivemmRho}). However, 
sometimes it may have a certain influence, when combined with another 
source of randomization (e.g., the random almost cyclic control used in 
the simulations), as shown in lines 1,9 and 7-8 of Table \ref{table:3Dx0PositivemmRho}). 
In order to draw stronger conclusions, more simulations are needed.  

The relaxation parameters seem to contribute significantly to 
the speed of convergence: the greater they are, the faster the convergence, 
but this dependence is not purely monotone (lines 21-24 of Table  \ref{table:3Dx0PositivemmRho}). On the other hand, because of \eqref{eq:ActiveStep} one may 
expect that the greater the product $\epsilon_1\epsilon_2$, the faster the convergence, but at least in our setting this has not 
been observed. In this connection, an interesting and unexplained phenomenon is 
 described in lines 41-44 of Table \ref{table:3Dx0PositivemmRho}: we have  
$\epsilon_1+\epsilon_2>2$ but still the algorithm works. However, when we tried to take  $\epsilon_1\geq 2.1$ the program crashed. 

Regarding the control, sometimes (e.g., line 13 comparing to line 15 in 
Table  \ref{table:3Dx0PositivemmRho}) the cyclic control 
leads to faster convergence, but not always (line 1 comparing to line 5 in Table  \ref{table:3Dx0PositivemmRho}). From the comparisons of line 23 to line 32 
we see that the speed can also be quite similar. 
However, since the comparison was  limited (not only because of the number of 
simulations and the way the control was created, but also because we used a 
concrete type of zero-convex functions), and since we sometimes had some problems 
with the random number generator (and hence with the random  
vector generator), one has to be careful when drawing 
  conclusions regarding 
the advantage of one control over the other. 

In the higher dimensional version of the original 3D setting, the data in Table \ref{table:HigherDimensions} show that the algorithm works in this case too. This is 
of course not really surprising, in view of  Theorem \ref{thm:resiliency}, 
but still one has to be careful since in some rare cases \eqref{eq:wp_wa} 
can be violated (when the dimension grows usually 
$0.47=R-r<\|a_j-p\|$ even if each component of $a_j-p$ is very small). The  last lines 
of this table show that the algorithm 
works when the dimension is 3 and locations of the molecules are random   
(with the exception that we always took $p=(0,0,0)$ and $a_{\ell}=(0,0,3.5)$). 
The value $5\cdot 10^6$ that sometimes appear there means that no feasible point 
was found after $5\cdot 10^6$ iterations.

\section{Further discussion}\label{sec:FurtherDiscussion}
This section concludes the paper with further discussion of certain issues. 
In Subsection \ref{subsec:different} we discuss the possibility of inner perturbations.  
In Subsection \ref{subsec:Compare} we compare briefly the SSP approach for solving 
the CFP to other possible optimization approaches. In Subsection  \ref{subsec:ApproximateMinimization} we explain how the results of this paper can be used for approximate minimization. Finally, in Subsection \ref{subsec:Conclude} we mention several open problems and lines for further investigation. 

\subsection{Two alternative presentations of perturbation
resilience\label{subsec:different}}

The perturbation resilience result established in Theorem \ref{thm:resiliency}
above looks different in nature than the results described in
\cite{ButnariuDavidiHermanKazantsev,CensorDavidiHerman,CDHST2013,DavidiPhD,HermanDavidi}. 
There the perturbed iterative step was of the form 
\begin{equation}
x_{n+1}=A_{n}(x_{n}+b_{n}) \label{eq:x}%
\end{equation}
for some sequence of perturbation vectors $b_{n}$ and a sequence of
algorithmic operators $A_{n}:\mathbb{R}^{d}\rightarrow\mathbb{R}^{d}$. That
format enabled the creation of \textit{superiorized algorithms} that use the
perturbations proactively in order to achieve an additional aim while being
guaranteed that the original convergence of the algorithm is preserved. 
In contrast, in 
\eqref{eq:iterative}, at least when $\Omega=H$, the perturbed iterative step
has the form
\begin{equation}
x_{n+1}=A_{n}(x_{n})+\widetilde{b}_{n} \label{eq:super}%
\end{equation}
where
\begin{equation}
\widetilde{b}_{n}:=\left\{
\begin{array}
[c]{ll}%
b_{n}, & \text{if }g_{i(n)}(x_{n})>0,\\
0, & \text{if }g_{i(n)}(x_{n})\leq0.
\end{array}
\right.  \label{eq:equ-1}%
\end{equation}
In this form the perturbations express the computational (numerical) error
resulting from a non-ideal computation of $A_{n}(x_{n})$. However, it is
possible to obtain a convergence result in the spirit of (\ref{eq:x}) by
modifying an argument which appears in \cite[p. 541]%
{ButnariuDavidiHermanKazantsev}. Indeed, define the sequence of operators
$A_{n}: H\to H$
\begin{equation}
A_{n}(x):=\left\{
\begin{array}
[c]{ll}%
x-\lambda_{n}\displaystyle{\frac{g_{i(n)}(x)}{\Vert t_{n}\Vert^{2}}}t_{n}, &
\text{if }g_{i(n)}(x)>0,\\
x, & \text{if }g_{i(n)}(x)\leq0,
\end{array}
\right.  \label{eq:equ-2}%
\end{equation}
and a new algorithmic sequence of vectors%
\begin{equation}
\left\{
\begin{array}
[c]{l}%
z_{0}=A_{0}(x_{0}),\\
z_{n+1}=A_{n+1}(z_{n}+\widetilde{b}_{n}).
\end{array}
\right.  \label{eq:equ-3}%
\end{equation}
Using this notation we obtain the following proposition.

\begin{proposition}
\label{prop:InternalPerturbations} Let $\{x_{n}\}_{n=0}^{\infty}$ be a sequence
in $\Omega:=H$, generated by \eqref{eq:super}, and let $\{z_{n}\}_{n=0}^{\infty}$
be the sequence generated by (\ref{eq:equ-3}) with $\widetilde{b}_{n}$ and
$A_{n}(x)$ defined as in (\ref{eq:equ-1}) and (\ref{eq:equ-2}), respectively.
Suppose that $\{b_{n}\}_{n=1}^{\infty}$ is a sequence in $H$ satisfying
$\lim_{n\to\infty} b_{n}=0$. If $\{x_{n}\}_{n=0}^{\infty}$ converges weakly to
some $x_{\ast}$, then also $\{z_{n}\}_{n=0}^{\infty}$ converges weakly to
$x_{\ast}$ and vice versa. If $\{x_{n}\}_{n=0}^{\infty}$ converges strongly,
then $\{z_{n}\}_{n=0}^{\infty}$ converges strongly to the same limit and vice versa.
\end{proposition}

\begin{proof}
It follows, by induction, that%
\begin{equation}
x_{n+1}=z_{n}+\widetilde{b}_{n},\quad\forall n\in\mathbb{N}\cup\{0\}.
\label{eq:xnznbn}%
\end{equation}
Since $\lim_{n\to\infty} b_{n}=0$ we have $\lim_{n\rightarrow\infty}\Vert
b_{n}\Vert=0$. Thus $\lim_{n\rightarrow\infty}\Vert\widetilde{b}_{n}\Vert=0$.
Since $x_{\ast}$ is the weak limit of the sequence $\{x_{n}\}_{n=0}^{\infty}$ it
follows from \eqref{eq:xnznbn} that the weak $\lim_{n\rightarrow\infty}z_{n}$
exists and equals $x_{\ast}$. A similar reason implies that if $\{x_{n}\}_{n=0}^{\infty}$ 
converges strongly, then $\{z_{n}\}_{n=0}^{\infty}$ converges 
strongly to the same limit. Finally, the reverse directions, namely the
convergence of $\{x_{n}\}_{n=0}^{\infty}$ from the convergence of 
$\{z_{n}\}_{n=0}^{\infty}$, hold by the same reasoning.
\end{proof}

Since under the conditions of Theorem  \ref{thm:resiliency} any sequence
$\{x_{n}\}_{n=0}^{\infty}$, generated by Algorithm \ref{alg:perturbed-csp}, 
converges (weakly or strongly) to a point in the feasible set, then so does the
sequence $\{z_{n}\}_{n=0}^{\infty}$, generated by (\ref{eq:equ-3}). Thus, Theorem  
\ref{thm:resiliency} and Proposition  \ref{prop:InternalPerturbations} can be
used in the superiorization methodology by allowing the algorithmic sequence
to have the form defined in either \eqref{eq:super} or \eqref{eq:equ-3}.

\subsection{Comparison with other methods for solving the CFP}\label{subsec:Compare}
A possible way to solve the CFP is to formulate it as a minimization problem. 
For example, one can define a function $f:\Omega\to\R$ by $f(x):=\max\{\sup_{j\in J}g_{j}(x),0\}$ and solve the problem 
\begin{equation}\label{eq:min_f}
\min_{x\in\Omega}f(x)
\end{equation}
which has an optimal value $0$, given that the CFP is feasible (i.e., $C$ from \eqref{eq:nonempty} is nonempty). One may 
use many of the known methods to solve the above optimization problem,
e.g., the usual subgradient descent methods. However, these methods 
require the functions $g_j$ to be convex (so that $f$ will be convex), while 
in Algorithm  \ref{alg:perturbed-csp} we allow the functions $g_j$ to be zero-convex 
(in \cite{SolodovZavriev1998} the target function $f$ may be nonconvex, but 
no convergence to the optimal value is proved unless $f$ is 
strongly convex, and additional assumptions are needed in the analysis). 
In addition, each iteration in Algorithm  \ref{alg:perturbed-csp} (in 
\eqref{eq:iterative}) depends only on one function $g_j$, while in \eqref{eq:min_f} 
each iteration depends on all the functions due to the definition of $f$. 
This dependence makes each iteration computationally demanding when many functions 
are involved. In addition, the convergence result described in Theorem  \ref{thm:resiliency} 
holds in a quite general setting, while in the case of \eqref{eq:min_f}, if 
for instance one allows perturbations, then some restrictions are imposed 
(e.g., the underlying  $\Omega$ should be compact or the function $f$ 
should have a set of sharp minima \cite{NedicBertsekas2010}). On the other hand, 
in the case of \eqref{eq:min_f} one may have convergence even if the 
problem is not feasible, while we do not 
know what happens in this case for the sequence generated by 
Algorithm  \ref{alg:perturbed-csp}.

\subsection{Approximate minimization}\label{subsec:ApproximateMinimization}
The results of this paper can be used for approximate 
minimization of a quasiconvex function $f:\Omega\to \R$ on $C=\bigcap_{j\in J} g_j^{\leq 0}$.  
More precisely, assume that  $\alpha\in \R$ 
is a known upper bound on $\inf_C f$ and that we want to find an $\alpha$-approximate minimizer  of $f$, that is, a point $x\in C$ satisfying $f(x)\leq \alpha$. 
Denote $g_{-1}:=f-\alpha$ (still quasiconvex and hence 0-convex),  and assume that all the 
assumptions of Theorem \ref{thm:resiliency} are satisfied with respect to the functions 
$g_j, j\in J\bigcup\{-1\}$. Assume also that $-1\notin J$. Apply Algorithm \ref{alg:perturbed-csp} with these functions. 
Theorem \ref{thm:resiliency} ensures that we will obtain a point $x$ belonging to the set  
$C\bigcap g_{-1}^{\leq 0}$, that is, a point $x\in C$ satisfying $f(x)\leq \alpha$, 
as required. The above generalizes \cite[Corollary 6.11(i)]{BauschkeCombettes2001} 
from the setting of approximate minimization of a convex function using 
the SSP without perturbations.

\subsection{Open questions and issues for future investigation}\label{subsec:Conclude} 
We conclude the paper by listing several open questions and lines 
for further investigation. 

Regarding weakening Theorem \ref{thm:resiliency}, we ask if 
Condition \ref{cond:boundedness} can be removed,  
for instance, when the growth of $\|t_n\|$ is not too large. 
Second, can the weak convergence  be extended to strong convergence without 
the assumption that the interior of $F$ is nonempty? Third, can the assumption \eqref{eq:control} on the control be relaxed to random (repetitive) controls? 
or at least can it be modified to other controls such as the 
most violated constraint control? In this connection, it may also be interesting 
to say something about the growth rate of the sequence $\{L_j\}_{j\in J}$ 
from \eqref{eq:control} when $J$ is infinite (see also \cite{CensorChenPajoohesh2011}).  

Another question is to obtain explicit error estimates for 
the speed of convergence in Theorem \ref{thm:resiliency}. It is not so easy to find such 
explicit estimates in many closely related theorems in the literature (theorems in 
which Fej\'er monotonicity is used for proving convergence), and unfortunately, 
so far this is true also regarding Theorem \ref{thm:resiliency}.  However, if 
one imposes additional assumptions, then it seems reasonable 
to believe that actually such explicit estimates (in fact, strong convergence in 
a linear rate) can be obtained. This belief is based on analogous results 
in the literature (for projection algorithms) in the case where the 
subsets $C_j$, $j\in J$ ($J$ is finite) are boundedly linearly regular 
(in particular, hyperplanes) \cite[Sections 5-7]{bb96}, or 
 certain affine subspaces \cite[Theorem 5.7.8]{BBL1997}, or a  
 Slater-type condition is satisfied and the control is almost cyclic 
 \cite[Theorem 7.18]{bb96}, \cite[Theorem 2]{DePierroIusem1988}. 

A different approach to the question of explicit estimates is to follow the analysis in 
\cite{Zaslavski2013a,Zaslavski2013b} in which one does not obtain a convergence 
result but rather obtains explicit time complexity estimates for 
approximate solutions. More precisely, given a tolerance parameter $\epsilon>0$ 
and an upper bound $\delta>0$ on the perturbations, one finds explicitly 
an iteration index $k_0$ and a point $x_{k_0}\in H$ such that $g_j(x_{k_0})\leq \epsilon$ 
for all $j$, under certain assumptions on the setting 
(e.g., there are finitely many convex and Lipschitz functions $g_j$ and  
the control is cyclic). In this case it may 
happen that $x_{k_0}$ is located far away from the intersection $C=\bigcap_{j\in J}C_j$, 
but perhaps under some additional assumptions on the subsets $C_j$, e.g., that there exists 
$\Delta\in (0,1]$ such that $\{x\in H|\,\, g_j(x)\leq \Delta,\,\,\,\forall j\in J\}\neq\emptyset$ (a Slater-type condition) 
and that $C$ is bounded, one can also find an explicit upper bound for 
$d(x_{k_0},C)$ as done in \cite[Section 6]{Zaslavski2013a}. 
The closely related analysis given in \cite{DePierroIusem1988}, which preceded 
\cite{Zaslavski2013a,Zaslavski2013b}, seems to help too in this direction.

The computational results of Section \ref{sec:ComputationalResults}, and, in particular, the  
improvement in the speed of convergence when the relaxation parameters grow, deserve an explanation. 
An intuitive and incomplete explanation of this phenomenon is the 
geometric interpretation of the algorithm which is closely related to Remark  \ref{rem:Geometric0-Convex}  and Figure \ref{fig:0ConvHyperplane}. 

It may be of interest to study further the notion of zero-convexity in various ways. 
One possibility is to follow the path of many works related to quasiconvex 
programming or generalized convexity, e.g.,  \cite{ADSZ2010,CambiniMartein2009-book,CM-LV-1998-handbook,HKS2005-handbook,Martinez-Legaz1988}, and in particular to study notions of duality in this context. 
Another possibility is to consider spaces which are more general than Hilbert spaces. 
As said after Definition \ref{def:0subdifferential}, the notion of zero-convexity can be generalized almost 
word for word to arbitrary normed spaces and beyond. This fact and the analysis of 
the proof of Theorem \ref{thm:resiliency} cause us to believe that 
 (perhaps slight variations of) this theorem  hold in the case where the setting is certain Banach spaces  (uniformly convex Banach spaces having a weakly continuous duality mapping),  certain Bregman distances (thus generalizing \cite{Kiwiel1998}), and certain Riemannian manifolds 
 (thus generalizing \cite{BentoMelo2012}). Indeed, the proof of Theorem \ref{thm:resiliency} is constructed 
in such a way that the assumption that $H$ is Hilbert does not appear in too 	
many places and at least in some places where it  appears there are more general results in 
the literature which can be used, as noted after Lemma \ref{lem:BrowderOpialBB}. 

In addition to generalizations of the above type, we believe that the notion of 
zero-convexity can be modified (and be useful) so it will cover zero-level-sets composed of 
a disjoint union of closed and convex subsets, and also to certain $\beta$-level sets 
instead of just $0$-level sets. 

Finally, it would be interesting to consider algorithmic schemes different from 
Algorithm \ref{alg:perturbed-csp} that will not be anymore sequential, but 
rather mixed or parallel (taking into account blocks, strings, weighted sums), and also 
to obtain results in the infeasible case (where the intersection $C$ from \eqref{eq:nonempty} is empty).

\begin{acknowledgements}
We thank Luba Tetruashvili  for many 
helpful comments. We thank Benar Svaiter for helpful 
discussions and suggestions, in particular, for his ideas regarding the
alternative proof of Proposition  \ref{prop:0ConvCharacterize}%
\eqref{item:LevelSet} mentioned in Remark  \ref{rem:ConevxZeroLevelImplyConvex}. 
Thanks are also due to Jefferson Melo, Alfredo Iusem, Simeon Reich, Alex Segal, 
and Mikhail Solodov  for helpful remarks, in particular regarding some of the references, 
and to Claudia Sagastiz\'abal for an helpful discussion on some general aspects of the paper. 
Finally, we greatly appreciate the constructive and insightful comments of
three anonymous reviewers and the Associate Editor which helped
us improve the paper.
This research was supported by the United States-Israel Binational Science Foundation 
(BSF) grant number 200912, by the US Department of Army award number W81XWH-10-1-0170, 
and by a special postdoc fellowship from IMPA.\bigskip\ 
\end{acknowledgements}

% BibTeX users please use one of
%\bibliographystyle{spbasic}      % basic style, author-year citations
%\bibliographystyle{spmpsci}      % mathematics and physical sciences
%\bibliographystyle{spphys}       % APS-like style for physics
%\bibliography{}   % name your BibTeX data base

\begin{thebibliography}{}
%
% and use \bibitem to create references. Consult the Instructions
% for authors for reference list style.
%

\bibitem {AlberIusemSolodov1998} Alber, Y.I.,  Iusem, A.N., Solodov, M.V.:  
\emph{On the projected subgradient method for nonsmooth convex optimization in
a Hilbert space}, Mathematical Programming \textbf{81} (1998), 23--35.

\bibitem {AmbroProdi-Book-1993} Ambrosetti, A., Prodi, G.: \emph{A {P}rimer
of {N}onlinear {A}nalysis}, Cambridge University Press, Cambridge, UK, 1993.

\bibitem {AmeBerEpp-Algs-1999} Amenta, A.B.,  Bern, M.W.,  Eppstein, D.:
\emph{{Optimal point placement for mesh smoothing}}, Journal of Algorithms
\textbf{30} (1999), 302--322.


\bibitem {Aurenhammer} Aurenhammer, F.: \emph{{V}oronoi diagrams - a survey of a
fundamental geometric data structure}, ACM Computing Surveys \textbf{3}
(1991), 345--405.

\bibitem{ADSZ2010}
Avriel, M., Diewert, W.E., Schaible, S.,  Zang, I.: \emph{Generalized
  concavity}, Classics in Applied Mathematics, Vol. 63, SIAM, Philadelphia,
  PA, USA, 2010, an unabridged republication of the work first published by Plenum
  Press, 1988. 

\bibitem{BGJ2013}
Barron, E.N., Goebel, R.,  Jensen, R.R.: \emph{Quasiconvex functions and
  nonlinear {PDE}s}, Transactions of the American Mathematical Society \textbf{365} (2013), 
  4229--4255. 


\bibitem{BarronLiu1997}
Barron, E.N., Liu, W.: \emph{Calculus of variations in {$L^\infty$}}, 
Applied Mathematics and Optimization \textbf{35} (1997), 237--263. 

\bibitem {bb96} Bauschke  H.H., Borwein, J.M.:  \emph{On projection algorithms
for solving convex feasibility problems}, SIAM Review \textbf{38} (1996), 367--426.

\bibitem{BBL1997}
Bauschke, H.H. , Borwein, J.M. , Lewis, A.S.: \emph{The method of cyclic
  projections for closed convex sets in Hilbert space}, Contemporary Mathematics 
  \textbf{204} (1997), 1--38.


\bibitem{BauschkeCombettes2001}
H.H. Bauschke and P.L. Combettes, \emph{A weak-to-strong convergence principle
  for {F}ej\'er-monotone methods in {H}ilbert spaces}, Mathematics of Operations
  Research \textbf{26} (2001), 248--264.

  
\bibitem {BauschkeCombettes2011} Bauschke, H.H.,  Combettes, P.L.: 
\emph{Convex Analysis and Monotone Operator Theory in Hilbert Spaces},
Springer, New York, NY, USA, 2011.

\bibitem{BCK2006}
Bauschke, H.H., Combettes, P.L., Kruk, S.G.: \emph{Extrapolation algorithm for
  affine-convex feasibility problems}, Numerical Algorithms \textbf{41} (2006),
  239--274


\bibitem {bss-3ed-2006} Bazarra, M.S.,  Sherali, H.,  Shetty, C.M.:  
\emph{Nonlinear Programming: Theory and Algorithms}, Third edition,
Wiley-Interscience, Hoboken, NJ, USA, 2006.

\bibitem {Ben-IsraelMond1986}  Ben-Israel, A., Mond, B.: \emph{What is
invexity?} Journal of the Australian Mathematical Society, Series B,
\textbf{28} (1986), 1--9.

\bibitem {BtEgN2009} Ben-Tal, A.,  El-Ghaoui, L.,  Nemirovskii, A.:  \emph{Robust
Optimization}, Princeton University Press, Princeton, NJ, USA, 2009.


\bibitem{BentoMelo2012}
 Bento, G.C., Melo, J.G.: \emph{Subgradient method for convex feasibility on
  {R}iemannian manifolds}, Journal of Optimization Theory and Applications 
  \textbf{152} (2012), 773--785. 
  

\bibitem {BertsekasNedicOzdaglar2003} Bertsekas, D.P., Nedi\'{c}, A. Ozdaglar, A.E.: 
 \emph{Convex Analysis and Optimization}, Athena Scientific, Belmont,
MA, USA, 2003.

\bibitem {BS00} Bonnans, J.F., Shapiro, A.:  \emph{Perturbation Analysis of
Optimization Problems}, Springer-Verlag, New York, NY, USA, 2000.


\bibitem {BorweinLewis-book-2006} Borwein, J.M., Lewis, A.S.: 
\emph{Convex Analysis and Nonlinear Optimization: Theory and Examples}, 2nd ed., 
Springer, New York, NY, USA, 2006.

\bibitem {BorweinZhu1999} Borwein, J.M., Zhu, Q.J.:  \emph{A survey of
subdifferential calculus with applications}, Nonlinear Analysis: Theory,
Methods and Applications \textbf{38} (1999), 687--773.

\bibitem {Browder1967} Browder, F.E.:  \emph{Convergence theorems for sequences
of nonlinear operators in Banach spaces}, Mathematische Zeitschrift
\textbf{100} (1967), 201--225.

\bibitem {ButnariuCensorGurfilHadar} Butnariu, D., Censor, Y., Gurfil, P., Hadar, E.: \emph{On the behavior of subgradient projections methods for convex
feasibility problems in Euclidean spaces}, SIAM Journal on Optimization
\textbf{19} (2008), 786--807.

\bibitem{ButnariuCensorReich1997} 
Butnariu, D., Censor, Y.,  Reich, S.: \emph{Iterative averaging of entropic
  projections for solving stochastic convex feasibility problems},
  Computational Optimization and Applications \textbf{8} (1997), 21--39.


\bibitem {ButnariuDavidiHermanKazantsev} Butnariu, D., Davidi, R., Herman, G.T.,  
 Kazantsev, I.G.:  \emph{Stable convergence behavior under summable
perturbations of a class of projection methods for convex feasibility and
optimization problems}, IEEE Journal of Selected Topics in Signal Processing
\textbf{1} (2007), 540--547.

\bibitem{ButnariuIusemBook} 
Butnariu, D., Iusem, A.N.:  \emph{Totally convex functions for fixed point
  computation and infinite dimensional optimization}, Kluwer Academic
  Publishers, Dordrecht, The Netherlands, 2000.
  
\bibitem {BRZ_InexactBregman} Butnariu, D., Reich, S.,  Zaslavski, A.J.:  
\emph{Convergence to fixed points of inexact orbits for {B}regman-monotone
operators and for nonexpansive operators in {B}anach spaces}, In: Fixed Point
Theory and its Applications, H. Fetter Natansky et al. (Editors), Yokohama
Publishers 2006, pp. 11--33.

\bibitem {ButnariuReichZaslavsky} Butnariu, D., Reich, S.,  Zaslavski, A.J.: \emph{Asymptotic behavior of inexact
orbits for a class of operators in complete metric spaces}, Journal of Applied
Analysis \textbf{13} (2007), 1--11.

\bibitem{CambiniMartein2009-book}
Cambini, A., Martein, L.: \emph{Generalized Convexity and Optimization},
  Lecture Notes in Economics and Mathematical Systems, Vol. 616,
  Springer-Verlag, Berlin, Germany, 2009. 

\bibitem{Cegielski2012}
A. Cegielski, \emph{Iterative methods for fixed point problems in {H}ilbert
  spaces}, Springer-Verlag, Berlin, Heidelberg, Germany, 2012.


\bibitem {cccdh11} Censor, Y.,  Chen, W.,  Combettes, P.L.,  Davidi, R., 
  Herman, G.T.: \emph{On the effectiveness of projection methods for convex
feasibility problems with linear inequality constraints}, Computational
Optimization and Applications \textbf{51 }(2012), 1065--1088.

\bibitem{CensorChenPajoohesh2011}
Y. Censor, W. Chen, and H. Pajoohesh, \emph{Finite convergence of a subgradient
  projections method with expanding controls}, Applied Mathematics \&
  Optimization \textbf{64} (2011), 273--285.

\bibitem {CensorDavidiHerman} Censor, Y., Davidi, R., Herman, G.T.:  
\emph{Perturbation resilience and superiorization of iterative algorithms},
Inverse Problems \textbf{26} (2010), 065008 (12pp).

\bibitem{CDHST2013}
Y. Censor, R. Davidi, G.T. Herman, R.W. Schulte, and L. Tetruashvili,
  \emph{Projected subgradient minimization versus superiorization}, Journal of
  Optimization Theory and Applications, accepted for publication (2013). 
  
\bibitem {CensorLent} Censor, Y., Lent, A.: \emph{An iterative row-action
method for interval convex programming}, Journal of Optimization Theory and
Applications \textbf{34} (1981), 321--353.

\bibitem {cl82} Censor, Y., Lent, A.: \emph{Cyclic subgradient projections}, Mathematical
Programming \textbf{24} (1982), 233--235.

\bibitem {CensorSegal2006} Censor, Y., Segal, A.: \emph{Algorithms for the
quasiconvex feasibility problem}, Journal of Computational and Applied
Mathematics \textbf{185} (2006), 34--50.

\bibitem {CensorZenios} Censor, Y., Zenios, S.A.: \emph{Proximal minimization
algorithm with {$D$}-functions}, Journal of Optimization Theory and
Applications \textbf{73} (1992), 451--464.

\bibitem {CensorZeniosBook} Censor, Y., Zenios, S.A.: \emph{Parallel
Optimization: Theory, Algorithms, and Applications}, Oxford University Press,
New York, NY, USA, 1997.

\bibitem {CGTV-IJRNC-2003} Chesi, G., Garulli, A., Tesi, A., Vicino, A.: 
\emph{Characterizing the solution set of polynomial systems in terms of
homogeneous forms: An {LMI} approach}, International Journal of Robust and
Nonlinear Control \textbf{13} (2003), 1239--1257.


\bibitem{CSKM2013}
Chiu, S.N., Stoyan,  D.,  Kendall, W.S., Mecke, J.: \emph{Stochastic Geometry
  and its Applications}, Third Ed., John Wiley \& Sons, Ltd. Chichester, UK, 2013.

\bibitem {Clarke1983} Clarke, F.H.:  \emph{Optimization and Nonsmooth Analysis},
Wiley Interscience, New York, NY, USA, 1983.

\bibitem {CSLW1998} Clarke, F.H.,  Ledyaev, Y.S.,  Stern, R.J.  Wolenski, P.R.:  
\emph{Nonsmooth Analysis and Control Theory}, Springer, New York, NY, USA, 1998.

\bibitem {Combettes1996} Combettes, P.L.:  \emph{The convex feasibility problem
in image recovery}, Advances in Imaging and Electron Physics \textbf{95}
(1996), 155--270.

\bibitem {Combettes2001} Combettes, P.L.: \emph{Quasi-{F}ej\'{e}rian analysis of
some optimization algorithms}, In: Inherently Parallel Algorithms in
Feasibility and Optimization and their Applications, D. Butnariu, Y. Censor
and S. Reich (Editors), Elsevier Science Publishers, Amsterdam, The
Netherlands, 2001, pp. 115--152.



\bibitem {CorvellecFlam} Corvellec, J.-N., Fl\aa m, S.D.:  \emph{Non-convex
feasibility problems and proximal point methods}, Optimization Methods and
Software \textbf{19} (2004), 3--14.


\bibitem{CM-LV-1998-handbook}
Crouzeix, J.-P., Martinez-Legaz, J.-E., Volle, M.: (eds.), \emph{Generalized
  Convexity, Generalized Monotonicity: Recent Results}, Nonconvex Optimization
  and its Applications, Vol. 27,  Kluwer Academic Publishers, Dordrecht, 
  The Netherlands, 1998.

\bibitem {DJL2009} Daniilidis, A., Jules, F., Lassonde, M.:  
\emph{Subdifferential characterization of approximate convexity: the lower
semicontinuous case}, Mathematical Programming \textbf{116} (2009), 115--127.

\bibitem {DavidiPhD} Davidi, R.:  \emph{Algorithms for Superiorization and Their
Applications to Image Reconstruction}, Ph.D. Thesis, Department of Computer
Science, The Graduate Center, The City University of New York (CUNY), New
York, NY, USA, 2010. Available at: \url{http://www.dig.cs.gc.cuny.edu/~rdavidi/Dissertation\_RanDavidi.pdf}.

\bibitem {dhc09} Davidi, R.,  Herman, G.T.,  Censor, Y.:  
\emph{Perturbation-resilient block-iterative projection methods with
application to image reconstruction from projections}, International
Transactions in Operational Research \textbf{16 }(2009), 505--524.

\bibitem {DES1982} Dembo, R.S.,  Eisenstat, S.C,  Steihaug, T.:  \emph{Inexact
Newton methods}, SIAM Journal on Numerical Analysis \textbf{19} (1982), 400--408.

\bibitem{DePierroIusem1988}
 De Pierro, A.R.,  Iusem, A.N.: \emph{A finitely convergent ``row-action'' method
  for the convex feasibility problem}, Applied Mathematics and Optimization
  \textbf{17} (1988), 225--235.
  
\bibitem {Dixit1976} Dixit, A.K.:  \emph{Optimization in Economic Theory}, Oxford
University Press,New York, NY, USA, 1976.




\bibitem {Eckstein1998} Eckstein, J.:  \emph{Approximate iterations in
Bregman-function-based proximal algorithms}, Mathematical Programming
\textbf{83} (1998) 113--123.


\bibitem {Epp-MSRI-2005} Eppstein, D.: \emph{{Quasiconvex programming}}, in:
Combinatorial and Computational Geometry (J.E. Goodman, J. Pach, and
E. Welzl, Editors), MSRI Publications, Vol. 52, Cambridge University Press, 
New York, NY, USA, 2005, pp. 287--331.

\bibitem {Epp-TALG-2004-qaba} Eppstein, D.: \emph{Quasiconvex analysis of
backtracking algorithms}, ACM Transactions on Algorithms \textbf{2} (2006), 492--509.


\bibitem{GTL1995}
Gerstein, M., Tsai, J.,  Levitt, M.: \emph{The volume of atoms on the protein
  surface: Calculated from simulation, using {V}oronoi polyhedra}, Journal 
  of Molecular Biology \textbf{249} (1995), 955--966.
  
\bibitem {GoebelReich} Goebel, K., Reich, S.: \emph{Uniform Convexity,
Hyperbolic Geometry, and Nonexpansive Mappings}, Marcel Dekker Inc., New York,
NY, USA, 1984.


\bibitem{GPF1997}
Goede, A., Preissner, R., Fr\"ommel, C.,: \emph{Voronoi cell: New method for
  allocation of space among atoms: Elimination of avoidable errors in
  calculation of atomic volume and density}, Journal of Computational Chemistry
  \textbf{18} (1997), 1113--1123.
  
  
\bibitem{GG-book-1981} 
Gohberg, I. Goldberg, S.: \emph{Basic Operator Theory}, Birkh\"auser, Boston, MA, USA,
  1981.
  
\bibitem {VoronoiWeb} Gold, C.: \emph{The {V}oronoi {W}eb {S}ite}, 2008, \url{http://www.voronoi.com/wiki/index.php?title=Main_Page}.

\bibitem {GreenbergPierskalla}Greenberg, H.P., Pierskalla, W.P.:  
\emph{Quasi-conjugate functions and surrogate duality}, Cahiers du Centre
d`Etudes de Recherche Operationnelle \textbf{15} (1973), 437--448.

\bibitem{Gromicho1998-book}
Gromicho, J.A.: \emph{Quasiconvex Optimization and Location Theory}, Applied
  Optimization, Vol. 9, Kluwer Academic Publishers, Dordrecht, The Netherlands, 1998.


\bibitem{HKS2005-handbook}
Hadjisavvas, N., Koml{\'o}si, S., Schaible, S.: (eds.), \emph{Handbook of
  Generalized Convexity and Generalized Monotonicity}, Nonconvex Optimization
  and its Applications, Vol. 76, Springer-Verlag, New York, NY, USA, 2005. 


\bibitem {Hanson1981} Hanson, M.A.:  \emph{On sufficiency of the Kuhn-Tucker
conditions}, Journal of Mathematical Analysis and Applications \textbf{80}
(1981), 545--550.

\bibitem {HenrLassCSM2004} Henrion, D., Lasserre, J.B.:  \emph{Solving
nonconvex optimization problems}, IEEE Control Systems Magazine \textbf{24}
(2004), 72--83.

\bibitem {HenrLass-Bookchap-2005} Henrion, D., Lasserre, J.B.: \emph{Detecting global optimality
and extracting solutions in {G}lopti{P}oly}, In: Positive Polynomials in
Control (D. Henrion and A. Garulli, Editors), Lecture Notes on Control and
Information Sciences, Vol. 312, Springer-Verlag, Berlin, Germany, 2005, pp. 293--310.

\bibitem {HenrLassTAC2012} Henrion, D., Lasserre, J.B.: \emph{Inner approximations for polynomial
matrix inequalities and robust stability regions}, IEEE Transactions on
Automatic Control \textbf{57} (2012), 1456--1467.

\bibitem {HermanDavidi} Herman, G.T.,  Davidi, R.: \emph{Image reconstruction
from a small number of projections}, Inverse Problems \textbf{24} (2008), 045011.

\bibitem {hgdc12} Herman, G.T.,  Gardu\~{n}o, E.,  Davidi, R.,  Censor, Y.:  
\emph{Superiorization: An optimization heuristic for medical physics}, Medical
Physics \textbf{39 }(2012), 5532--5546.

\bibitem {higham96} Higham,, N.J.:  \emph{Accuracy and Stability of Numerical
Algorithms}, Society of Industrial and Applied Mathematics (SIAM),
Philadelphia, PA, USA, 1996.

\bibitem {HildebrandKoppe2013} Hildebrand, R.,  K{\"{o}}ppe, M.: \emph{A
new {L}enstra-type algorithm for quasiconvex polynomial integer minimization
with complexity $2^{O(n\log n)}$}, Discrete Optimization \textbf{10} (2013),  69--84.

\bibitem{HP1995-handbook}
Horst, R., Pardalos, P.M.: (eds.), \emph{Handbook of Global Optimization},
  Nonconvex Optimization and its Applications, Vol. 2, Kluwer Academic
  Publishers, Dordrecht, The Netherlands, 1995. 


\bibitem{HorstTuy-book-1990} 
Horst, R., Tuy, H.: \emph{Global Optimization: Deterministic Approaches},
  Springer-Verlag, Berlin, Heidelberg, Germany, 1990.
    
\bibitem {IusemMoledo} Iusem, A.N.,  Moledo, L.: \emph{A finitely convergent
method of simultaneous subgradient projections for the convex feasibility
problem}, Computational and Applied Mathematics \textbf{5} (1986), 169--184.

\bibitem {IusemOtero} Iusem, A.N.,   Otero, R.G.:  \emph{Inexact versions of
proximal point and augmented {L}agrangian algorithms in {B}anach spaces},
Numerical Functional Analysis and Optimization \textbf{22} (2001), 609--640.

\bibitem {Ivanov2004} Ivanov, M.:  \emph{Sequential representation formulae for
$G$-subdifferential and Clarke subdifferential in smooth Banach spaces},
Journal of Convex Analysis \textbf{11} (2004), 179--196.


\bibitem{KWCKLBK2006}
 Kim, C.-M.,  Won, C.-I., Cho, Y., Kim, D., Lee, S., Bhak, J., Kim, D.-S.:  
  \emph{Interaction interfaces in proteins via the {V}oronoi diagram of atoms},
  Computer-Aided Design \textbf{38} (2006), 1192--1204.
  
  
\bibitem{Kiwiel1998}
Kiwiel, K.C.: \emph{Generalized {B}regman projections in convex feasibility
  problems}, Journal of Optimization Theory and Applications \textbf{96} (1998), 
  139--157.



\bibitem {Kiwiel2004} Kiwiel, K.C.:  \emph{Convergence of approximate and incremental
subgradient methods for convex optimization}, SIAM Journal on Optimization
\textbf{14 }(2004), 807--840.


\bibitem {KopeckaReemReich} Kopeck\'{a}, E., Reem, D., Reich, S.: \emph{Zone
diagrams in compact subsets of uniformly convex spaces}, Israel Journal of
Mathematics \textbf{188} (2012), 1--23.

\bibitem {LassSIOPT2001} Lasserre, J.B.:  \emph{Global optimization with
polynomials and the problem of moments}, SIAM Journal on Optimization
\textbf{11} (2001), 796--817.

 
\bibitem{Martinez-Legaz1988}
 Mart{\'{\i}}nez-Legaz, J.-E.: \emph{Quasiconvex duality theory by generalized
  conjugation methods}, Optimization \textbf{19} (1988), 603--652.
  
  
\bibitem {MartinezLegazSach} Mart\'{\i}nez-Legaz, J.E., Sach, P.H.:  \emph{A
new subdifferential in quasiconvex analysis}, Journal of Convex Analysis
\textbf{6} (1999), 1--11.

\bibitem{MonteiroSvaiter2013} Monteiro, R. D. C., Svaiter, B. F.: 
\emph{An accelerated hybrid proximal extragradient method for convex optimization and its implications to second-order methods}, 
SIAM Journal on Optimization \textbf{23}  (2013), 1092--1125.

\bibitem {Mordukhovich1976} Mordukhovich, B.S.:  \emph{Maximum principle in
problems of time optimal control with nonsmooth constraints}, Journal of
Applied Mathematics and Mechanics \textbf{40} (1976), 960--969.

\bibitem {NedicBertsekas2010} Nedi\'{c}, A., Bertsekas, D.P.:  \emph{The effect
of deterministic noise in subgradient methods}, Mathematical Programming,
Series A \textbf{125} (2010), 75--99.

\bibitem {NgaiLucThera} Ngai, H.V.,  Luc, D.T., Thera, M.:  \emph{Approximate
convex functions}, Journal of Nonlinear and Convex Analysis \textbf{1} (2000), 155--176.

\bibitem {OBSC} Okabe, A., Boots,B., Sugihara, K.,  Chiu, S.N.:  \emph{Spatial
{T}essellations: {C}oncepts and {A}pplications of {V}oronoi {D}iagrams}, John
Wiley \& Sons Ltd., Chichester, UK, 2000, Second edition.

\bibitem {Opial1967} Opial, Z.:  \emph{Weak convergence of the sequence of
successive approximations for nonexpansive mappings}, Bulletin of the American
Mathematical Society \textbf{73} (1967), 591--597.

\bibitem {OstrowskiStability} Ostrowski, A.M.:  \emph{The round-off stability of
iterations}, Zeitschrift f\"{u}r Angewandte Mathematik und Mechanik
\textbf{47} (1967), 77--81.

\bibitem {OteroIusem2012} Otero, R.G.,  Iusem, A.: \emph{Fixed-point methods
for a certain class of operators}, Journal of Optimization Theory and
Applications \textbf{159} (2013), 656--672.

\bibitem {PallaschkeRolewicz1997} Pallaschke, D.,  Rolewicz, S.:  
\emph{Foundations of Mathematical Optimization: Convex Analysis Without
Linearity}, Kluwer Academic Publishers, Norwell, MA, USA, 1997.
  
\bibitem {scottetal10} Penfold, S.N.,  Schulte, R.W.,  Censor, Y., 
 Rosenfeld, A.B.:  \emph{Total variation superiorization schemes in proton computed
tomography image reconstruction}, Medical Physics \textbf{37} (2010), 5887--5895.

\bibitem{Penot1998} 
Penot, J.-P.:  \emph{Are generalized derivatives useful for generalized convex
  functions?} In: Generalized convexity, Generalized Monotonicity: Recent Results
  (J.E. Martinez-Legaz J.-P. Crouzeix and M. Volle, eds.), Kluwer Academic
  Publishers, Dordrecht, The Netherlands, 1998, p. 3--59.

\bibitem{Pinter-book-1996} 
Pint\'er, J.: \emph{Global Optimization in Action: Continuous and Lipschitz
  Optimization--Algorithms, Implementations, and Applications}, Kluwer
  Academic, Dordrecht, The Netherlands, 1996.
  
 \bibitem {Plastria} Plastria, F.: \emph{Lower subdifferentiable functions and
their minimization by cutting planes}, Journal of Optimization Theory and
Applications \textbf{46} (1985), 37--53.


\bibitem{PRZ2008} 
Pustylnik, E., Reich, S., Zaslavski, A.J.: \emph{Inexact orbits of
  nonexpansive mappings}, Taiwanese Journal of Mathematics \textbf{12} (2008), 1511--1523.

\bibitem{PRZ2009} 
Pustylnik, E., Reich, S., Zaslavski, A.J.: \emph{Inexact infinite products of nonexpansive mappings}, Numerical
  Functional Analysis and Optimization. \textbf{30} (2009), 632--645.

\bibitem {Razumichin1987} Razumichin, B.S.:  \emph{Classical Principles and
Optimization Problems}, D. Reidel Publishing Company, Dordrecht, Holland, 1987.

  
\bibitem {ReemISVD09} Reem, D.: \emph{An algorithm for computing {V}oronoi
diagrams of general generators in general normed spaces}, Proceedings of the
Sixth International Symposium on {V}oronoi Diagrams in Science and Engineering
(ISVD 2009), Copenhagen, Denmark, pp. 144--152.

\bibitem {ReemGeometricStabilityArxiv} Reem, D.: \emph{The geometric stability
of {V}oronoi diagrams with respect to small changes of the sites}, (2011),
Complete version in arXiv 1103.4125 (last updated: April 6, 2011), 
Extended abstract in: Proceedings of the
27th Annual Symposium on Computational Geometry (SoCG 2011), Paris, France,
pp. 254--263.

\bibitem {ReemBregmanII2012} Reem, D.: \emph{The {B}regman distance without the
{B}regman function {II}}, Contemporary Mathematics \textbf{568} (2012), 213--223.

\bibitem {Rockafellar1970} Rockafellar, R.T.:  \emph{Convex {A}nalysis},
Princeton University Press, Princeton, NJ, USA, 1970.

\bibitem {Rockafellar1976} Rockafellar, R.T.: \emph{Monotone operators and the proximal
point algorithm}, SIAM Journal on Control and Optimization \textbf{14} (1976), 877--898.

\bibitem {Rockafellar1979} Rockafellar, R.T.: \emph{Directionally Lipschitzian functions
and subdifferential calculus}, Proceedings of the London Mathematical Society
\textbf{39} (1979), 331--355.

\bibitem{RKCPK2005}
Ryu, J., Kim, D., Cho, Y., Park, R.,  Kim, D.-S.: \emph{Computation of molecular
  surface using {E}uclidean {V}oronoi diagram}, Computer-Aided Design and
  Applications \textbf{2} (2005), 439--448.


\bibitem {Soleimani-damaneh2010} Soleimani-damaneh, M.:  \emph{Nonsmooth
optimization using Mordukhovich's subdifferential}, SIAM Journal on Control
and Optimization \textbf{48} (2010), 3403--3432.

\bibitem {SolodovSvaiter2000} Solodov, M.V.,  Svaiter, B.F.:  \emph{An inexact
hybrid generalized proximal point algorithm and some new results in the theory
of {B}regman functions}, Mathematics of Operations Research \textbf{51}
(2000), 479--494.

\bibitem {SolodovSvaiter2001} Solodov, M.V.,  Svaiter, B.F.: \emph{A unified framework for some
inexact proximal point algorithms}, Numerical Functional Analysis and
Optimization \textbf{22} (2001), 1013--1035.

\bibitem {SolodovZavriev1998} Solodov, M.V., Zavriev, S.K.: \emph{Error
stability properties of generalized gradient-type algorithms}, Journal of
Optimization Theory and Applications \textbf{98} (1998), 663--680.


\bibitem {SvaiterPC2011} Svaiter, B.F.: \emph{Personal Communication} (2011).

\bibitem {SvaiterPC2012} Svaiter, B.F.: \emph{Personal Communication} (2012).

\bibitem {Tuncel2010} Tuncel, L.:  \emph{Polyhedral and Semidefinite Programming
Methods in Combinatorial Optimization}, A co-publication of the American
Mathematical Society (AMS) and the Fields Institute for Research in
Mathematical Sciences, Providence, RI, USA, 2010.
  
\bibitem{VanTiel1984} 
Van Tiel, J.: \emph{{C}onvex {A}nalysis: {A}n {I}ntroductory {T}ext}, John Wiley
  and Sons, Chichester, UK, 1984.


\bibitem{Zaslavski2013a}
Zaslavski, A.J.: \emph{Subgradient projection algorithms for convex feasibility
  problems in the presence of computational errors}, Journal of Approximation Theory
  \textbf{175} (2013), 19--42. 
    
\bibitem{Zaslavski2013b}
Zaslavski, A.J.: \emph{Subgradient projection algorithms and approximate
  solutions of convex feasibility problems}, Journal of Optimization Theory and 
  Applications \textbf{157} (2013), 803--819. 
  


\end{thebibliography}

% Non-BibTeX users please use

\end{document}